\documentclass[12pt]{article}

\usepackage{amsmath}
\usepackage{mathtools}
\usepackage{amssymb}
\usepackage{enumerate}
\usepackage{here}
\usepackage{mathrsfs}
\usepackage{multirow}
\usepackage[sort,compress]{cite}
\usepackage{dsfont}
\usepackage{color}
\usepackage{fullpage}

\usepackage{newfloat}
\DeclareFloatingEnvironment[
    fileext=lob,
    listname=List of Boxes,
    name=Box,
    placement=tbhp,
]{mybox}

\everymath{\displaystyle}

\newtheorem{theorem}{Theorem}
\newtheorem{corollary}[theorem]{Corollary}
\newtheorem{problem}[theorem]{Problem}
\newtheorem{proposition}[theorem]{Proposition}

\DeclareMathOperator*{\col}{col}

\def\tkZ{\{t_k\}_{k=0}^\infty}

\DeclareMathOperator{\He}{Sym}
\DeclareMathOperator*{\diag}{diag}

\DeclareMathOperator{\eps}{\varepsilon}

\DeclareMathOperator*{\argmin}{arg\ min}

\newenvironment{proof}{{\it Proof :~}}{\hfill$\diamondsuit$\\}

\providecommand{\bblue}[1]{\color{black}{#1}\color{black}\hspace{0pt}}
\providecommand{\blue}[1]{\color{black}{#1}\color{black}\hspace{0pt}}

\begin{document}


\title{Co-design of aperiodic sampled-data min-jumping rules for linear impulsive, switched impulsive and sampled-data systems}

\author{Corentin Briat\thanks{corentin@briat.info, http://www.briat.info}}
\date{}

\maketitle`

\begin{abstract}
\blue{Co-design conditions for the design of a jumping-rule and a sampled-data control law for impulsive and impulsive switched systems subject to aperiodic sampled-data measurements are provided. Semi-infinite discrete-time Lyapunov-Metzler conditions are first obtained. As these conditions are difficult to check and generalize to more complex systems, an equivalent formulation is provided in terms of clock-dependent (infinite-dimensional) matrix inequalities. These conditions are then, in turn, approximated by a finite-dimensional optimization problem using a sum of squares based relaxation. It is proven that the sum of squares relaxation is non conservative provided that the degree of the polynomials is sufficiently large. It is emphasized that acceptable results are obtained for low polynomial degrees in the considered examples.}\\

\end{abstract}

\section{Introduction}

\blue{The co-design problem in control theory amounts to simultaneously determining two components of the considered control system. Important examples are, for instance, the co-design of a sampling-rule together with a corresponding controller \cite{Anta:10,Seuret:11b}, the co-design of a continuous control law and a switching rule for a switched system \cite{Liberzon:03,Allerhand:13}, the co-design of communication protocols and controllers in networked control systems \cite{Hristu:99,Peters:15} or the co-design of multiple continuous/sampled-data controllers along with a switching rule in order to stabilize a given system \cite{Morse:96,Prieur:11}.

In the context of linear switched systems, a popular switching rule is the so-called "min-switching rule" \cite{Geromel:06b,Geromel:06d,Allerhand:13}. In this setup, a quadratic Lyapunov function is consider for each mode and the mode is chosen in real-time in a way that makes the value of the Lyapunov function minimum. It can be shown that this corresponds to finding a certain partitioning of the state-space. An alternative approach \cite{Kruszewski:11} is based on the use of a common quadratic Lyapunov function for all the subsystems and the switching rule is chosen such that the derivative of the Lyapunov function is minimum at any time. Although slightly different, these approaches share the underlying assumption that the state/output of the system is continuously measured, which may be impractical from an implementation point of view or may lead to undesired chattering. A potential workaround to these issues would be the consideration of sampled measurements instead of continuous-time ones. Establishing event-triggered or self-triggered sampling strategies are possible solutions but have been almost exclusively considered in the sampled-data control of continuous-time systems. The co-design problem adds a layer of difficulty since one has to, on the top of the switching-/sampling-rule, simultaneously design a controller. In the context of switched systems, that was recently solved in \cite{Allerhand:13} using timer-/clock-dependent Lyapunov functions which leads to infinite-dimensional stability and stabilization conditions that are affine in the matrices of the system and, hence, amenable to be solved using polynomial methods and semidefinite programming techniques \cite{Parrilo:00,sostools3,Putinar:93}. Analogous conditions for impulsive and switched systems were also proposed in \cite{Briat:11l,Briat:12h,Briat:13b} using looped-functionals and in \cite{Briat:13d,Briat:14f,Briat:15i} using clock-dependent Lyapunov functions.

Most of the results on the design of switching rules pertain on the assumption of continuous measurements of the state/output. The case of sampled periodic/aperiodic measurements has been relatively few studied in the literature in spite of its importance. \bblue{For instance, a sampled-data switching rule is designed for stabilizing switched affine systems is considered in \cite{Hetel:13}. The approach is based on minimizing the derivative of the Lyapunov function at each sampling instants, which is another an alternative approach to the min-switching rule. Perhaps more closely, a periodic sampled-data switching rule for switched systems is developed in \cite{Deaecto:13b} using a min-switching approach. However, the approach does not consider aperiodic sampling, the presence of state jumps and the design of a sampled-data control law.} Since time-varying delays, jitter and packet loss induced by communication channels in communication networks destroy the periodicity of the arrival times of the sampled measurements, making the sampling scheme apparently aperiodic (see e.g. \cite{Briat:11d,Yuan:16}), it seemed important to incorporate this phenomenon in the current setup. Several methods have been proposed to deal with aperiodic sampling. One is based on the reformulation of a sampled-data system as an input-delay system that can be analyzed using Lyapunov-Krasovskii functionals; see e.g. \cite{Fridman:04}. A second approach is a robust one where the sampling operator is considered as an uncertainty and the interconnection is analyzed using robust analysis techniques \cite{Mirkin:07,Fujioka:09b,Kao:14}. Direct discrete-time approaches have also been considered in \cite{Fujioka:09a,Oishi:10}. Finally, the hybrid formulation \cite{Goebel:09} of a sampled-data system has been considered in various works and analyzed using Lyapunov functionals \cite{Naghshtabrizi:08}, Looped-functionals \cite{Seuret:12,Briat:11l,Briat:14b} and clock-/timer- dependent Lyapunov functions \cite{Goebel:09,Briat:13d,Briat:15i,Briat:16c}.

The objective of this paper is to address an analogous problem for linear impulsive systems, a general class of hybrid systems that encompasses switched and sampled-data systems as particular instances \cite{Goebel:12,Briat:13d}. We consider here a set-up where both a \emph{min-jumping rule} (an analogue to the min-switching rule considered in the context of switched systems) and a continuous-time.sampled-data state-feedback controller in the context of unpredictable sampling instants. This has to be contrasted with the setup where the future sampling instant is known (periodic sampling or event-/self- triggered sampling strategies) and the setup where measurements are continuous. Such a scenario occurs, for instance, in networked control systems where actuation decisions are made upon reception of measurements from sensors, which may happen aperiodically due to the presence of sampling, jitter, delays and packet loss \cite{Hespanha:07,Zhang:16c}. The advantage of such an approach is twofold: it rules out the chattering phenomenon of the continuous-time methods and is realistic from an implementation viewpoint since no continuous measurement is assumed. It is assumed here that the measurements from sensors arrive at discrete time instants which are assumed to satisfy a mild range dwell-time condition \cite{Briat:11l,Briat:13d}. Hence, both periodic and aperiodic measurements are considered and can be easily defined in a way to incorporate jitter, small delays and self/event-triggered sampling mechanisms \cite{Yuan:16}. The proposed approach is then applied to impulsive systems and switched-impulsive systems which can be used to model networked control systems subject to delays and communication outages, systems controlled by multiple controllers or systems with limited actuation resources \cite{Hespanha:07}. To the author's knowledge, this is the first time that conditions for the co-design of asynchronous switching/jumping rules and a sampled-data state-feedback controller are obtained for impulsive and switched impulsive systems.

Sufficient stabilization conditions for the co-design of min-jumping rule and, possibly, a sampled-data controller are first stated as discrete-time Lyapunov-Metzler conditions \cite{Geromel:06d,Heemels:17}. These conditions are extended to the co-design of min-jumping/switching rule and sampled-data controllers for impulsive switched systems. The conditions are stated as robust linear matrix inequalities with a scalar uncertainty at the exponential, which makes them difficult to check; \cite{Briat:11l,Briat:12h} for a similar discussion. These conditions are also difficult to consider when additional uncertainties are involved in the system expression, when the system is not time-invariant or even nonlinear. Looped-functionals \cite{Seuret:12,Briat:11l,Briat:12h,Briat:13b} and clock-dependent Lyapunov functions \cite{Allerhand:11,Briat:13d,Briat:14f,Briat:15f,Briat:15i,Xiang:15a,Xiang:16} were introduced as a workaround to equivalently represent a discrete-time stability condition into a form that is affine in the matrices of the system. The price to pay is the infinite dimensionality of the affine conditions meaning that they cannot be directly checked. However, finite-dimensional approximations can be obtained by relying on  sum of squares (SOS) \cite{Parrilo:00,Briat:13d,Briat:14f,Briat:15f}, Handelman's Theorem \cite{Handelman:88,Briat:16c,Scherer:06,Kamyar:15,Le:18} or discretization methods \cite{Allerhand:11}. In this regard, the infinite-dimensional conditions have the advantage of being readily applicable for design purposes, robustness/performance analysis; see e.g.  \cite{Allerhand:11,Briat:11l,Seuret:12,Briat:12h,Briat:13b,Briat:13d,Briat:14f,Briat:15f,Briat:15i,Xiang:15a,Xiang:16}. In this paper, SOS relaxations are considered as they often lead to more efficient finite-dimensional conditions than the others \cite{Briat:15f}. For all the obtained co-design conditions, converse results are obtained meaning that if the original Lyapunov-Metzler conditions are feasible, then there exists a solution to the SOS program. In this regard, considering polynomial can be seen as non restrictive in the present cases. Several examples illustrate the efficiency of the proposed conditions.}


\noindent\textbf{Notations.} The set of symmetric matrices of dimension $n$ is denoted by $\mathbb{S}^n$ and for $A,B\in\mathbb{S}^n$, $A\preceq B$ means that $A-B$ is negative semidefinite. The cone of symmetric positive (semi)definite matrices of dimension $n$ is denoted by ($\mathbb{S}^n_{\succeq0}$) $\mathbb{S}^n_{\succ0}$. The $n$-dimensional vector of ones is denoted by $\mathds{1}_n$.

\bblue{\section{Problem formulation and motivations}\label{sec:general}

We consider in this paper the following general class of switched impulsive systems with multiple jump maps and sampled-data state-feedback:
\begin{equation}\label{eq:mainsystg}
\begin{array}{rcl}
  \dot{x}(t)&=&A_{\sigma(t)}x(t)+B_{\sigma(t)}u(t),\ t\in \mathbb{R}_{\ge0}\backslash\tkZ\\
  u(t)&=&K_{\sigma(t_k^+),\sigma(t_k)}^1x(t_k)+K_{\sigma(t_k^+),\sigma(t_k)}^2u(t_k),\ t\in(t_k,t_{k+1}]\\
  x(t_k^+)&=&J_{\sigma(t_k^+),\sigma(t_k)}x(t_k),\ k\in\mathbb{Z}_{\ge0}
\end{array}
\end{equation}
where $x,x_0\in\mathbb{R}^n$ and $u\in\mathbb{R}^{m}$ are the state of the system, the initial condition and the control input, respectively. The notation $x(t_k^+)$ is defined as $\textstyle x(t_k^+):=\lim_{s\downarrow t_k}x(s)$ and we have that $\textstyle x(t_k)=\lim_{s\uparrow t_k}x(s)$, i.e. the trajectories are left-continuous. The signal $\sigma(t)\in\{1,\ldots,N\}$ is assumed to be piecewise constant and to only change value on $\tkZ$ where the sequence $\tkZ$ satisfies the range dwell-time condition  $T_k:=t_{k+1}-t_k\in[T_{min},T_{max}]$, $0<T_{min}\le T_{max}<\infty$, for all $k\in\mathbb{Z}_{\ge0}$. It is assumed throughout this paper that this sequence is not known a priori and is decided, for example, at the sensor level through some event-based algorithm that is not known from the controller and scheduling side. Interestingly, this may occur in the case of periodic sampling which is subject to a sensor-to-controller time-varying delay or data loss; see e.g. \cite{Briat:11d,Yuan:16}. Note, however, that at any time $t$, the sequence $\{t_k\}_{k\in \kappa(t)}$ with $\kappa(t):=\left\{k\in\mathbb{Z}_{\ge0}:t_k\le t\right\}$ is known.

In this paper, we will address the following general co-design problem for the system \eqref{eq:mainsystg}:
\begin{problem}
  Assuming aperiodic sampled-data measurements $x(t_k)$ satisfying a range dwell-time condition, solve the co-design problem for the system \eqref{eq:mainsystg}, that is the goal is to design
  \begin{enumerate}[(a)]
    \item a sampled-data state-dependent sampled-data switching law $\sigma$, and
    \item a mode-dependent sampled-data control law $u$
  \end{enumerate}
  such that the closed-loop system is asymptotically stable under the same range dwell-time constraint.
\end{problem}
Co-design problems are, in general, not easy as they can be often non-convex and, therefore, difficult to solve, unless in some particular cases. In the present paper, we aim at proposing solutions which can be checked using convex programming techniques.

To clarify the ideas, we briefly give now several interesting examples covered by the aforementioned general setup. The first one is the \emph{sampled-data controller co-design for LTI systems}. This problem is captured by the system
\begin{equation}
\begin{array}{rcl}
  \dot{x}(t)&=&Ax(t)+Bu(t),\ t\in \mathbb{R}_{\ge0}\backslash\tkZ\\
  u(t)&=&K_{\sigma(t_k^+),\sigma(t_k)}^1x(t_k)+K_{\sigma(t_k^+),\sigma(t_k)}^2u(t_k),\ t\in(t_k,t_{k+1}]
\end{array}
\end{equation}
and the aim is to design $N$ sampled-data controllers (i.e. the gains $K_{i,j}^1$, $K_{i,j}^2$, $i,j=1,\ldots,N$) together with a selecting rule $\sigma$ deciding on what controllers to use upon sampled-measurements arrivals.

The second example concerns the control of jump systems
\begin{equation}
\begin{array}{rcl}
  \dot{x}(t)&=&A(t)x(t),\ t\in \mathbb{R}_{\ge0}\backslash\tkZ\\
  x(t_k^+)&=&J_{\sigma(t_k)}x(t_k),\ k\in\mathbb{Z}_{\ge0}
\end{array}
\end{equation}
where the aim is to select which jump matrix to use upon sampled-measurements arrivals. Note that there is no co-design problem here unless some jump matrices can be partially or fully designed. For instance, if $J_i=:J_i^0+J_i^1X_i$ for some $i$, where the $X_i$'s are matrices to be designed and the others are fixed and known. In such a case, the problem becomes a co-design problem which can be solved using convex optimization techniques. Note that this may not be the case when the matrices $J_i$ have a different structure.

Finally, the third example pertains on the co-design of sampled-data controllers along with a switching rule based on sampled-data measurements for switched systems. This problem is well captured by the following model
\begin{equation}
\begin{array}{rcl}
  \dot{x}(t)&=&A_{\sigma(t)}x(t)+B_{\sigma(t)}u(t),\ t\in \mathbb{R}_{\ge0}\backslash\tkZ\\
  u(t)&=&K_{\sigma(t_k^+),\sigma(t_k)}^1x(t_k)+K_{\sigma(t_k^+),\sigma(t_k)}^2u(t_k),\ t\in(t_k,t_{k+1}]\\
\end{array}
\end{equation}
where the goal is to design $N$ sampled-data controllers (i.e. the gains $K_{i,j}^1$, $K_{i,j}^2$, $i,j=1,\ldots,N$) together with a switching rule $\sigma$ for the switched system as well for selecting the right sampled-data controller.}

\section{Results for sampled-data impulsive systems}\label{sec:imp}

\bblue{We start with the results on impulsive systems. We first give some preliminaries, followed by the main theoretical results. Due to the structure of the conditions, a section pertaining on solving those conditions is provided.}

\subsection{Preliminaries}

Let us consider in this section the following impulsive system with multiple jump maps and sampled-data state-feedback:
\begin{equation}\label{eq:mainsyst_imp}
\begin{array}{rcl}
  \dot{x}(t)&=&Ax(t)+Bu(t), t\in \mathbb{R}_{\ge0}\backslash\tkZ\\
  u(t)&=&K_{\sigma(t_k^+)}^1x(t_k)+K_{\sigma(t_k^+)}^2u(t_{k}),\ t\in(t_k,t_{k+1}]\\
  x(t_k^+)&=&J_{\sigma(t_k^+)}x(t_k),\ k\in\mathbb{Z}_{\ge0}
\end{array}
\end{equation}
where $x(\cdot),x_0\in\mathbb{R}^n$, $u(\cdot)\in\mathbb{R}^m$ are the state of the system, the initial condition and the control input. The signal $\sigma(t)\in\{1,\ldots,N\}$ as well as the sequence $\tkZ$ satisfy the conditions stated in Section \ref{sec:general}. It is assumed throughout this paper that this sequence is not known a priori and is decided, for example, at the sensor level through some event-based algorithm that is not known from the controller and scheduling side. Interestingly, this may occur in the case of periodic sampling which is subject to a sensor-to-controller time-varying delay or data loss; see e.g. \cite{Briat:11d,Yuan:16}.
%

The objective of this section is to obtain co-design conditions for the simultaneous design of the state-feedback control gains $K_i^1\in\mathbb{R}^{m\times n},K_i^2\in\mathbb{R}^{m\times m}$, $i=1,\ldots,N$ and the jump scheduling law $\sigma(t)$ that is actuated only at the times in $\tkZ$ using only current state-information. In order to solve this problem, we first rewrite the impulsive sampled-data system \eqref{eq:mainsyst_imp} into an impulsive system with augmented state-space
\begin{equation}
  \begin{array}{rcl}
  \dot{\chi}(t)&=&\underbrace{\begin{bmatrix}
      A & B\\
      0 & 0
    \end{bmatrix}}_{\mbox{$\bar A$}}\chi(t),\ t\in \mathbb{R}_{\ge0}\backslash\tkZ\\
    \chi(t_k^+)&=&\underbrace{\begin{bmatrix}
      J_{\sigma(t_k^+)} & 0\\
      K^1_{\sigma(t_k^+)} & K^2_{\sigma(t_k^+)}
    \end{bmatrix}}_{\mbox{$\bar{J}_{\sigma(t_k^+)}$}}\chi(t_k),\ k\in\mathbb{Z}_{\ge0}
  \end{array}
\end{equation}
where $\chi(t):=\col(x(t),u(t))$, $\bar{J}_{i}=:\bar{J}^0_{i}+\bar{J}^1_{i}K_{i}$, $K_i:=[K_i^1\ \ K_i^2]$, $i=1,\ldots,N$, and
for which we propose the following \emph{min-jumping rule}
\begin{equation}\label{eq:rule1}
  \sigma(t_k^+)=\argmin_{i\in\{1,\ldots,N\}}\left\{\chi(t_k)^TP_i\chi(t_k)\right\}
\end{equation}
where the matrices $P_i\in\mathbb{S}^{n+m}_{\succ0}$ have to be designed. This rule is clearly inspired by the continuous min-switching law considered in \cite{Geromel:06b} in the context of switched systems.

The following result states a sufficient condition for the stability of the system \eqref{eq:mainsyst_imp} controlled with the min-jumping rule \eqref{eq:rule1}:
\begin{proposition}\label{prop:1}
  Let $0<T_{min}\le T_{max}<\infty$ and assume that there exist matrices $P_i\in\mathbb{S}^{n+m}_{\succ0}$, $i=1,\ldots,N$, a nonnegative matrix\footnote{A nonnegative matrix is a matrix containing nonnegative entries.} $\Pi\in\mathbb{R}^{N\times N}$ verifying $\mathds{1}_N^T\Pi=\mathds{1}_N^T$ and a scalar $\eps>0$ such that the condition
\begin{equation}\label{eq:0}
  \bar J_i^Te^{\bar A^T\theta}\left(\sum_{j=1}^N\pi_{ji}P_j\right)e^{\bar A\theta}\bar J_i-P_i +\eps I\preceq0
\end{equation}
holds for all $i=1,\ldots,N$ and all $\theta\in[T_{min},T_{max}]$.

Then, the system \eqref{eq:mainsyst_imp} controlled with the rule \eqref{eq:rule1} is asymptotically stable for any sequence $\tkZ$ satisfying the range dwell-time condition $t_{k+1}-t_k\in[T_{min},T_{max}]$.
\end{proposition}
\begin{proof}
  The proof of this result is based on the consideration of the linear discrete-time switched system
\begin{equation}
  \chi(t_{k+1})=e^{\bar AT_k}\bar{J}_{\sigma(t_k^+)}\chi(t_k)
\end{equation}
whose stability is equivalent to the original impulsive system. Using now the result in \cite{Geromel:06d} for discrete-time switched systems yields the result.
\end{proof}

The condition \eqref{eq:0} is not an LMI because of the products between $P_j$ and $\pi_{ji}$ but becomes so whenever the $\pi_{ji}$'s are chosen a priori. An approach based on stationary distributions of Markov processes is proposed in \cite{Geromel:06b,Geromel:06d}. Another one relies on the simplification of the matrix $\Pi$ (i.e. Lemma 1 in \cite{Geromel:06d}) or on an equivalent simpler condition \cite{Deaecto:13b}. However, these approaches do not directly apply to the current problem because of the uncertain parameter $\theta$ and the presence of the controller gains that need to be designed. It is unclear at this time whether there exists a similar way to solve this problem and a brute force approach may be necessary to appropriately select this matrix. \bblue{A possible solution would be the consideration of a $\theta$-dependent matrix $\Pi$. This is left for future research.}

\blue{It is also worth mentioning that the approach based on Lyapunov-Metzler inequalities has been favored here as it is computationally more appealing than an approach based on the $S$-procedure \cite{Heemels:17} which is known to be less conservative but involves more decision variables (i.e. $2N^2(N-1)$ instead of $N(N-1)$ for the current one) and more inequality constraints. In this regard, the practical advantage of the approach based on the $S$-procedure is unclear. In spite of that, the conditions in the above theorem can be easily be extended to this more general setting without any difficulty.}

Finally, when the state-feedback gains $K_i$ are to be computed, then the condition \eqref{eq:0} is not appropriate since it cannot be easily turned into a form that is convenient for design purposes. At last, the presence of the uncertain parameter $\theta$ at the exponential adds some computational complexity to the overall approach. However, this latter problem has now been extensively studied; see e.g. \cite{Heemels:10b,Allerhand:11,Briat:11l,Seuret:12,Briat:13d} and subsequent works of the same authors.

\subsection{Main results}

Inspired by the results in \cite{Briat:13d,Briat:14f}, the following co-design result is obtained:
\begin{theorem}\label{th:main_IMP}
Let $0<T_{min}\le T_{max}<\infty$. Then, the following statements are equivalent:
\begin{enumerate}[(a)]
  \item There exist matrices $P_i\in\mathbb{S}^{n+m}_{\succ0}$, $K_i\in\mathbb{R}^{m\times(n+m)}$, $i=1,\ldots,N$, and a nonnegative matrix $\Pi\in\mathbb{R}^{N\times N}$ verifying $\mathds{1}_N^T\Pi=\mathds{1}_N^T$ such that the condition \eqref{eq:0}
holds for all $i=1,\ldots,N$ and all $\theta\in[T_{min},T_{max}]$.
\item There exist differentiable matrix-valued functions $S_i:[0,T_{max}]\mapsto\mathbb{S}^{n+m}$, matrices $P_i\in\mathbb{S}_{\succ0}^{n+m}$, $K_i\in\mathbb{R}^{m\times(n+m)}$, $i=1,\ldots,N$, a nonnegative matrix $\Pi\in\mathbb{R}^{N\times N}$ verifying $\mathds{1}_N^T\Pi=\mathds{1}_N^T$  and a scalar $\eps>0$ such that the conditions
\begin{equation}\label{eq:1}
  -\dot{S}_i(\tau)+\bar A^TS_i(\tau)+S_i(\tau)\bar A\preceq0
\end{equation}
\begin{equation}\label{eq:2}
  -P_i+\bar J_i^TS_i(\theta)\bar J_i+\eps I\preceq0
\end{equation}
and
\begin{equation}\label{eq:3}
  \sum_{j=1}^N\pi_{ji}P_j-S_i(0)\preceq0
\end{equation}
hold for all $i=1,\ldots,N$, all $\tau\in[0,T_{max}]$ and all $\theta\in[T_{min},T_{max}]$.
\item There exist some differentiable matrix-valued functions $\tilde S_i:[0,T_{max}]\mapsto\mathbb{S}^{n+m}$, matrices $\tilde P_i\in\mathbb{S}_{\succ0}^{n+m}$, $U_i\in\mathbb{R}^{m\times (n+m)}$, $i=1,\ldots,N$, a nonnegative matrix $\Pi\in\mathbb{R}^{N\times N}$ verifying $\mathds{1}_N^T\Pi=\mathds{1}_N^T$  and a scalar $\eps>0$ such that the conditions
\begin{equation}
  \dot{\tilde S}_i(\tau)+\tilde S_i(\tau)\bar A^T+\bar A\tilde S_i(\tau)\preceq0
\end{equation}
\begin{equation}
\begin{bmatrix}
  -\tilde P_i & \star\\
  \bar J_i^0\tilde P_i+\bar J_i^1U_i & -\tilde S_i(\theta)
\end{bmatrix}\prec0
\end{equation}
and
\begin{equation}\label{eq:djsqsl}
-\diag_{j=1}^N\{\tilde P_j\}+V_i\tilde S_i(0)V_i^T\preceq0
\end{equation}
where $V_i=\col_{j=1}^N\{\pi_{ji}^{1/2}I_{n+m}\}$ hold for all $i=1,\ldots,N$, all $\tau\in[0,T_{max}]$ and all $\theta\in[T_{min},T_{max}]$. 
\end{enumerate}
Moreover, when the conditions of statement (c) hold, then the conditions in Proposition \ref{prop:1} hold with $K_i=U_i\tilde{P}_i^{-1}$ and  $P_i=\tilde{P}_i^{-1}$, $i=1,\ldots,N$. As a result,  the system \eqref{eq:mainsyst_imp} with the controller gains $K_i=U_i\tilde{P}_i^{-1}$, $i=1,\ldots,N$, and the rule \eqref{eq:rule1} with $P_i=\tilde{P}_i^{-1}$ is asymptotically stable for any sequence $\tkZ$ satisfying the range dwell-time condition $t_{k+1}-t_k\in[T_{min},T_{max}]$.
\end{theorem}
\begin{proof}
\bblue{\textbf{Proof that (a) implies (b).} Assume that the conditions in statement (a) hold and define
\begin{equation}
    S_i^\ast(\tau):=e^{\bar A^T\tau}\left[\sum_{j=1}^N\pi_{ji}P_j\right]e^{\bar A\tau}.
\end{equation}
We show now that $S_i=S_i^\ast$ verifies the conditions in statement (b). Differentiating $S_i^\ast(\tau)$ with respect to $\tau$ yields
\begin{equation}
  \dot{S}^\ast_i(\tau)=\bar A^TS_i^\ast(\tau)+S_i^\ast(\tau)\bar A
\end{equation}
and, hence, the condition \eqref{eq:1} is verified with $S_i=S_i^\ast$. Similarly, evaluating $S_i^\ast(\tau)$ at $\tau=0$ yields
\begin{equation}
  S_i^\ast(0)=\sum_{j=1}^N\pi_{ji}P_j
\end{equation}
and, therefore, the condition \eqref{eq:3} is satisfied with $S_i=S_i^\ast$. Finally, substituting $S_i=S_i^\ast$ in the LHS of \eqref{eq:2} yields
\begin{equation}
  \bar J_i^Te^{\bar A^T\theta}\left(\sum_{j=1}^N\pi_{ji}P_j\right)e^{\bar A\theta}\bar J_i-P_i+\eps I
\end{equation}
which is negative semidefinite since it was assumed that the statement (a) holds or, equivalently, that the condition \eqref{eq:0} holds. The proof is completed.}\\

\noindent\textbf{Proof that (b) implies (a).} Assume that the conditions of statement (b) hold. Then, integrating \eqref{eq:1} over $[0,\theta]$ yields \cite{Gajic:95}
\begin{equation*}
  e^{\bar A^T\theta}S_i(0)e^{\bar A\theta}-S_i(\theta)\preceq0
\end{equation*}
which together with \eqref{eq:3} implies that
\begin{equation*}
  e^{\bar A^T\theta}\left[\sum_{j=1}^N\pi_{ji}P_j\right]e^{\bar A\theta}-S_i(\theta)\preceq0.
\end{equation*}
By pre- and post-multiplying the above inequality by $\bar J_i^T$ and $\bar J_i$, respectively, and using \eqref{eq:2}, we finally obtain \eqref{eq:0}. The proof is completed.\\

\noindent\textbf{Proof that (b) is equivalent to (c).} This follows from the changes of variables $U_i=K_i\tilde{P}_i$, $\tilde P_i=P_i^{-1}$, $\tilde S_i(\tau)=S_i(\tau)^{-1}$ and Schur complements.
\end{proof}

The conditions stated in statement (c) are more convenient to consider than those stated in Proposition \ref{prop:1} since the uncertainty $\theta$ is not at the exponential anymore and the conditions are convex in the decision matrices $P_i$'s and $K_i$'s whenever the matrix $\Pi$ is chosen a priori. More specifically, these conditions are convex infinite-dimensional LMI conditions that can be solved in an efficient way using discretization methods \cite{Allerhand:11,Xiang:15a,Briat:15f} or sum of squares methods \cite{Parrilo:00,Briat:13d,Briat:14f,Briat:15f}, which can be both shown to be asymptotically exact. On the other hand, when $\Pi$ has to be designed, the problem becomes nonlinear and may be difficult to solve; see \cite{Geromel:06b,Geromel:06d} for some discussion on how to potentially solve such a problem.

When only the scheduling law needs to be designed, the following immediate corollary should be considered:
\begin{corollary}\label{cor:1}
Assume that $B=0$ in \eqref{eq:mainsyst_imp} and let $0<T_{min}\le T_{max}<\infty$. Then, the following statements are equivalent:
\begin{enumerate}[(a)]
  \item There exist matrices $P_i\in\mathbb{S}^n_{\succ0}$, $i=1,\ldots,N$, and a nonnegative matrix $\Pi\in\mathbb{R}^{N\times N}$ verifying $\mathds{1}_N^T\Pi=\mathds{1}_N^T$ such that the condition
  \begin{equation}
  J_i^Te^{A^T\theta}\left(\sum_{j=1}^N\pi_{ji}P_j\right)e^{A\theta}J_i-P_i\prec0
\end{equation}
holds for all $i=1,\ldots,N$ and all $\theta\in[T_{min},T_{max}]$.
\item There exist some differentiable matrix-valued functions $S_i:[0,T_{max}]\mapsto\mathbb{S}^n$, matrices $P_i\in\mathbb{S}_{\succ0}^n$, $i=1,\ldots,N$, a nonnegative matrix $\Pi\in\mathbb{R}^{N\times N}$ verifying $\mathds{1}_N^T\Pi=\mathds{1}_N^T$  and a scalar $\eps>0$ such that the conditions
\begin{equation}\label{eq:cor1}
  -\dot{S}_i(\tau)+A^TS_i(\tau)+S_i(\tau)A\preceq0
\end{equation}
\begin{equation}
  -P_i+J_i^TS_i(\theta)J_i+\eps I\preceq0
\end{equation}
and
\begin{equation}
  \sum_{j=1}^N\pi_{ji}P_j-S_i(0)\preceq0
\end{equation}
hold for all $i=1,\ldots,N$, all $\tau\in[0,T_{max}]$ and all $\theta\in[T_{min},T_{max}]$.
\end{enumerate}
Moreover, when the conditions of statement (b) are verified, then the system \eqref{eq:mainsyst} with $B=0$ controlled with the rule
\begin{equation}\label{eq:rule1b}
  \sigma(t_k^+)=\argmin_{i\in\{1,\ldots,N\}}\left\{x(t_k)^TP_ix(t_k)\right\}
\end{equation}
is asymptotically stable for any sequence $\tkZ$ satisfying the range dwell-time condition $t_{k+1}-t_k\in[T_{min},T_{max}]$.
\end{corollary}

\blue{\subsection{Numerical verification of the conditions and asymptotic exactness}

There are multiple ways for solving the conditions in Theorem \ref{th:main_IMP}. A first approach is to assume that the matrix-valued functions $S_i$ are piecewise linear, which would then result in a finite set of finite-dimensional LMI conditions; see e.g. \cite{Allerhand:11,Allerhand:13}. This approach is easy to use but may not be efficient because of its high computational complexity and its poor convergence properties as the number of pieces increase \cite{Briat:15f}. An alternative one is based on polynomial optimization techniques either relying on Handelman's Theorem \cite{Handelman:88,Briat:16c,Scherer:06,Kamyar:15,Le:18} or Putinar's Positivstellensatz \cite{Putinar:93,Parrilo:00,sostools3}. Both of these methods will result in a finite-dimensional semidefinite program which can be solved using standard solvers such as SeDuMi \cite{Sturm:01a}. \bblue{We propose to use here an approach based on sums of squares.} The conversion to a semidefinite program can be performed using the package SOSTOOLS \cite{sostools3} to which we input the SOS program corresponding to the considered conditions. We illustrate below how an SOS program associated with some given conditions can be obtained. We define the following intervals
\begin{equation}
  [0, T_{max}]=\left\{\tau\in\mathbb{R}:\ g(\tau):=\tau(T_{max}-\tau)\ge0\right\}.
\end{equation}
and
\bblue{\begin{equation}
  [T_{min}, T_{max}]=\left\{\tau\in\mathbb{R}:\ h(\tau):=(\tau-T_{min})(T_{max}-\tau)\ge0\right\}.
\end{equation}}
In what follows, we say that a symmetric polynomial matrix $\Theta(\cdot)$ is a sum of squares matrix (SOS matrix) or is SOS, for simplicity, if there exists a polynomial matrix $\Xi(\cdot)$ such that $\Theta(\cdot)=\Xi(\cdot)^{T}\Xi(\cdot)$.

\begin{mybox*}
\caption{SOS program associated with Corollary \ref{cor:1}, (b)}\label{box}
{\vspace{1mm}}
\noindent\fbox{
\parbox{\textwidth}{
Find symmetric matrix-valued polynomial functions $S_i,\Gamma_i,\Delta_i:\mathbb{R}\mapsto\mathbb{S}^{n+m}$, $i=1,\ldots,N$, constant symmetric matrices $P_i\in\mathbb{S}^{n+m}$, $i=1,\ldots,N$, a nonnegative matrix $\Pi\in\mathbb{R}^{N\times N}$ such that $\mathds{1}^T\Pi=\mathds{1}^T$ and a scalar $\eps>0$ such that
      \begin{itemize}
        \item $\Gamma_i(\cdot),\Delta_i(\cdot)$, $i=1,\ldots,N$, are SOS matrices,
        \item $P_i-\eps I$, $i=1,\ldots,N$, are SOS matrices,
        \item $\dot{S}_i(\tau)-\He[S_i(\tau)\bar A_i] -\Gamma_i(\tau)g(\tau)$ is an SOS matrix for all $i=1,\ldots,N$,
        \item $P_i-\bar{J}_i^TS_i(\theta)\bar{J}_i-\eps I -\Delta_i(\theta)h(\theta)$ is an SOS matrix for all $i=1,\ldots,N$,
        \item $-\sum_{j=1}^N\pi_{ji}P_j +S_i(0)$ is an SOS matrix for all $i=1,\ldots,N$.
    \end{itemize}}}
\end{mybox*}

\begin{proposition}[\cite{Parrilo:00}]
  A univariate polynomial matrix is positive semidefinite if and only if it is SOS.
\end{proposition}

We also need the following result:
\begin{proposition}[\cite{Scherer:06,Chesi:10b}]\label{prop:sp}
  Let us consider a univariate polynomial symmetric matrix $M(\cdot)$. Then, the following statements are equivalent:
 \begin{enumerate}[(a)]
    \item The matrix $M(\theta)$ is positive semidefinite for all  $\theta\in[T_{min}, T_{max}]$.
    \item There exists an SOS matrix $N(\cdot)$ such that the matrix $M(\theta)-N(\theta)h(\theta)$ is SOS.
 \end{enumerate}
\end{proposition}

The following result provides the sum of squares formulation of Theorem Corollary \ref{cor:1}, (b):
\begin{proposition}\label{prop:cor1}
  Let $\eps$ and $0<T_{min}\le T_{max}<\infty$ be given. The conditions of Corollary \ref{cor:1}, (b), are feasible with polynomial matrices $S_i$ if and only if the  sum of squares program in Box \ref{box} is feasible. Moreover, we have that the conditions in Corollary \ref{cor:1}, (a) hold.
\end{proposition}
\begin{proof}
Clearly the first statement in Box \ref{box} is equivalent to saying that the matrix-valued functions $\Gamma_i(\cdot),\Delta_i(\cdot)$ are positive semidefinite for all $i=1,\ldots,N$. The second statement is equivalent to saying that the matrices $P_i$ are positive definite for all $i=1,\ldots,N$. Since $\Gamma_i(\tau)\succeq0$ for all $\tau\in\mathbb{R}$ and $g(\tau)\ge0$ for all $\tau\in[0,\bar T]$, then, from Proposition \ref{prop:sp}, we can see that the second statement is equivalent to saying that $-\dot{S}_i(\tau)+\He[S_i(\tau)\bar{A}_i]\succeq0$ for all $\tau\in[0,\bar T]$, which coincides with the condition \eqref{eq:1} in Corollary \ref{cor:1}, (b). Similarly, the fourth statement is equivalent to saying that $-P_i+\bar{J}_i^TS_i(\theta)\bar{J}_i+\eps I\preceq0$ for all $\theta\in[T_{min},T_{max}]$, which coincides with the condition \eqref{eq:2} in Corollary \ref{cor:1}, (b). Finally, the last statement is equivalent to the condition \eqref{eq:3} in Corollary \ref{cor:1}, (b). This proves the first part of the result. The second part follows from the fact that (b) implies (a) in Corollary \ref{cor:1}.
\end{proof}

We have proved above that the conditions of Corollary \ref{cor:1}, (b) can be verified by solving a semidefinite program if we restrict ourselves to polynomial matrices. In this regard, we have proved that we have a sufficient SOS condition for assessing whether the statement (a) of Corollary \ref{cor:1} holds. The following result proves that we still have the equivalence between (a) and (b) when considering polynomial matrices.
\begin{proposition}[Asymptotic exactness]\label{prop:exact1}
 Let $0<T_{min}\le T_{max}<\infty$ be given. Assume that the conditions of Corollary \ref{cor:1}, (a) hold. Then, there exist $\eps>0$ and $d\in\mathbb{Z}_{\ge0}$ such that the SOS program in Box \ref{box} is feasible using polynomial matrices of degree $2d$.
\end{proposition}
\begin{proof}
  The general closed form for the $S_i$'s which solve the conditions in Corollary \ref{cor:1}, (b), is given by
  \begin{equation}
    S_i^\ast(\tau)=e^{\bar A^T\tau}\left[\sum_{j=1}^N\pi_{ji}P_j\right]e^{\bar A\tau}+\int_0^\tau \bar{J}_i^Te^{\bar A^T(\tau-s)}W_i(s)e^{\bar A(\tau-s)}\bar{J}_ids,\tau\in[0,T_{max}].
  \end{equation}
  where $W_i(s)\succeq0$ for all $i=1,\ldots,N$. Note that only the case $W_i\equiv0$ was considered in the proof of Theorem \eqref{th:main_IMP}. This expression is not a matrix polynomial except in some very particular cases. \bblue{However, this function is continuously differentiable and, hence, by virtue of the Stone-Weierstrass theorem, it can be uniformly approximated along with its derivative with respect to $\tau$ as closely as desired by a polynomial matrix.} This proves the result.
\end{proof}

\begin{theorem}\label{th:NCS1}
   Let $0<T_{min}\le T_{max}<\infty$ be given. Then, the following statements are equivalent:
   \begin{enumerate}[(a)]
     \item The conditions of Corollary \ref{cor:1}, (a) hold.
     \item There exist some scalars $\eps>0$ and $d\in\mathbb{Z}_{\ge0}$ such that the SOS program in Box \ref{box} is feasible using polynomial matrices of degree $2d$.
   \end{enumerate}
\end{theorem}
\begin{proof}
  The proof follows from Corollary \ref{cor:1} and Proposition \ref{prop:exact1}.
\end{proof}

%

\begin{mybox*}
\caption{SOS program associated with Theorem \ref{th:main_IMP}, (c)}\label{box2}
{\vspace{1mm}}
\noindent\fbox{
\parbox{\textwidth}{
Find symmetric matrix-valued polynomial functions $\tilde{S}_i,\Gamma_i,\Delta_i:\mathbb{R}\mapsto\mathbb{S}^{n+m}$, $i=1,\ldots,N$, constant symmetric matrices $\tilde{P}_i\in\mathbb{S}^{n+m}$, $i=1,\ldots,N$, constant matrices $U_i\in\mathbb{R}^{m\times(n+m)}$, $i=1,\ldots,N$, a nonnegative matrix $\Pi\in\mathbb{R}^{N\times N}$ such that $\mathds{1}^T\Pi=\mathds{1}^T$ and a scalar $\eps>0$ such that
      \begin{itemize}
        \item $\Gamma_i(\cdot),\Delta_i(\cdot)$, $i=1,\ldots,N$, are SOS matrices,
        \item $\tilde{P}_i-\eps I$, $i=1,\ldots,N$, are SOS matrices,
        \item $-\dot{\tilde{S}}_i(\tau)-\He[\tilde{S}_i(\tau)\bar A_i^T] -\Gamma_i(\tau)g(\tau)$ is an SOS matrix for all $i=1,\ldots,N$,
        \item $\begin{bmatrix}
          \tilde{P}_i-\eps I & (\bar{J}_i^0\tilde{P}_i+\bar{J}_i^1U_i)^T\\
          \bar{J}_i^0\tilde{P}_i+\bar{J}_i^1U_i & \tilde{S}_i(\theta)\bblue{-\Delta_i(\theta)h(\theta)}
        \end{bmatrix}$ is an SOS matrix for all $i=1,\ldots,N$,
        \item $\diag_{j=1}^N[\tilde{P}_j] - V_i\tilde{S}_i(0)V_i^T$ is an SOS matrix for all $i=1,\ldots,N$.
    \end{itemize}}}
\end{mybox*}

We have the analogous result for  Theorem \ref{th:main_IMP}, (c):
\begin{theorem}\label{th:NCS2}
   Let $0<T_{min}\le T_{max}<\infty$ be given. Then, the following statements are equivalent:
   \begin{enumerate}[(a)]
     \item The conditions of Theorem \ref{th:main_IMP}, (a) hold.
     \item There exist some scalars $\eps>0$ and $d\in\mathbb{Z}_{\ge0}$ such that the SOS program in Box \ref{box2} is feasible using polynomial matrices of degree $2d$.
   \end{enumerate}
\end{theorem}
\begin{proof}
  The proof is identical to that of Theorem \ref{th:NCS1}.
\end{proof}}

\section{Results for sampled-data impulsive switched systems}\label{sec:switched}

\subsection{Preliminaries}

Let us now consider the following class of switched impulsive systems with multiple jump maps and sampled-data state-feedback:
\begin{equation}\label{eq:mainsyst}
\begin{array}{rcl}
  \dot{x}(t)&=&A_{\sigma(t)}x(t)+B_{\sigma(t)}u(t),\ t\in \mathbb{R}_{\ge0}\backslash\tkZ\\
  u(t)&=&K_{\sigma(t_k^+),\sigma(t_k)}^1x(t_k)+K_{\sigma(t_k^+),\sigma(t_k)}^2u(t_k),\ t\in(t_k,t_{k+1}]\\
  x(t_k^+)&=&J_{\sigma(t_k^+),\sigma(t_k)}x(t_k),\ k\in\mathbb{Z}_{\ge0}
\end{array}
\end{equation}
where $x,x_0\in\mathbb{R}^n$ and $u\in\mathbb{R}^{m}$ are the state of the system, the initial condition and the control input, respectively. As before, the switching signal $\sigma:\mathbb{R}_{\ge0}\mapsto\{1,\ldots,N\}$ is assumed to be piecewise constant and to only change values on $\tkZ$.

The above system can be equivalently represented by the impulsive-switched system
\begin{equation}\label{eq:dksldskdl1}
  \begin{array}{rcl}
    \dot{\chi}(t)&=&\bar{A}_{\sigma(t)}\chi(t),\ t\in \mathbb{R}_{\ge0}\backslash\tkZ\\
   \chi(t_k^+)&=&\bar{J}_{\sigma(t_k^+),\sigma(t_k)}\chi(t_k),\ k\in\mathbb{Z}_{\ge0}
  \end{array}
\end{equation}
with $\chi=\col(x,u)$, $K_{j,i}:=[K_{j,i}^1\ K_{j,i}^2]$,
\begin{equation}\label{eq:dksldskdl2}
  \bar{A}_i=\begin{bmatrix}
      A_i & B_i\\
      0 & 0
    \end{bmatrix},\ \bar{J}_{j,i}=\begin{bmatrix}
      J_{j,i}& 0\\
      K^1_{j,i} & K^2_{j,i}
    \end{bmatrix}=:\bar{J}_{j,i}^0+\bar{J}_{j,i}^1K_{j,i}
\end{equation}
for $i,j=1,\ldots,N$ and for which we propose the jump rule:
\begin{equation}\label{eq:rule2}
  \sigma(t_k^+)=\argmin_{j\in\{1,\ldots,N\}}\left\{\chi(t_k)^T\bar{J}_{j,\sigma(t_k)}^TP_{j}\bar{J}_{j,\sigma(t_k)}\chi(t_k)\right\}.
\end{equation}

\subsection{Main results}

The following result states a sufficient condition for the stability of the system \eqref{eq:mainsyst} controlled with the min-jumping rule \eqref{eq:rule2}:
\begin{proposition}\label{prop:2}
   Let $0<T_{min}\le T_{max}<\infty$ and assume that there exist matrices $P_i\in\mathbb{S}^{n+m}_{\succ0}$ and a nonnegative matrix $\Pi\in\mathbb{R}^{N\times N}$ verifying $\mathds{1}_N^T\Pi=\mathds{1}_N^T$ such that the condition
        \begin{equation}\label{dsjkldsaldjasdjal}
                    e^{\bar A_i^T\theta}\left(\sum_{j=1}^N\pi_{ji}\bar J_{j,i}^TP_j\bar J_{j,i}\right)e^{\bar A_i\theta}-P_i\prec0
        \end{equation}
        holds for all $i=1,\ldots,N$ and all $\theta\in[T_{min},T_{max}]$.
%
%

Then, the system \eqref{eq:mainsyst} controlled with the min-jumping rule \eqref{eq:rule2} is asymptotically stable for any sequence $\tkZ$ satisfying the range dwell-time condition $t_{k+1}-t_k\in[T_{min},T_{max}]$.
\end{proposition}
\begin{proof}
   The proof consists of an extension of a proof in \cite{Geromel:06d} and is given below for completeness. Let $\sigma(t_k)=\sigma(t_{k-1}^+)=i$ and assume that the conditions of Proposition \ref{prop:2} hold. Define then the function
   \begin{equation}
   \begin{array}{rcl}
   \blue{V(\chi(t_k^+))}&:=&\min_{j\in{1,\ldots,N}}\left\{\chi(t_k^+)^TP_{j}\chi(t_k^+)\right\}\\
     &=&\min_{j\in{1,\ldots,N}}\left\{\chi(t_k)^T\bar{J}_{j,i}^TP_{j}\bar{J}_{j,i}\chi(t_k)\right\}
     \end{array}
   \end{equation}
   which is consistent with the rule \eqref{prop:2}. \blue{Note that the function $V(\chi(t_k^+))$ is radially unbounded since all the matrices $P_{j}$'s are positive definite.} Then, we have that
   \begin{equation}
   \begin{array}{rcl}
     \blue{V(\chi(t_k^+))}&=&\min_{\substack{\lambda\ge0\\ \mathds{1}_N^T\lambda=1}}\left\{\chi(t_k)^T\left(\sum_{j=1}^N\lambda_j\bar{J}_{j,i}^TP_{j}\bar{J}_{j,i}\right)\chi(t_k)\right\}\\
     &\le&\chi(t_{k-1}^+)^Te^{A_i^TT_k}\left(\sum_{j=1}^N\pi_{ji}\bar{J}_{j,i}^TP_{j}\bar{J}_{j,i}\right)e^{A_iT_k}\chi(t_{k-1}^+)\\
     &<&\chi(t_{k-1}^+)^TP_i\chi(t_{k-1}^+)=V(\chi(t_{k-1}^+))
   \end{array}
   \end{equation}
   provided that $T_k\in[T_{min},T_{max}]$. So, $V$ is a discrete-time Lyapunov function for the system
   \begin{equation}
     \chi(t_k^+)=\bar{J}_{\sigma(t_k^+),\sigma(t_k)}e^{\bar{A}_{\sigma(t_k)}T_k} \chi(t_{k-1}^+)
   \end{equation}
   controlled with the rule \eqref{eq:rule2}. Since the stability of the above discrete-time system is equivalent to that of \eqref{eq:mainsyst} under the  min-jumping rule \eqref{eq:rule2} then we can conclude that the system \eqref{eq:mainsyst} controlled with the min-jumping rule \eqref{eq:rule2} is asymptotically stable for any sequence $\tkZ$ satisfying the range dwell-time condition $t_{k+1}-t_k\in[T_{min},T_{max}]$. This completes the proof.
\end{proof}

As for Proposition \ref{prop:1}, the conditions in Proposition \ref{prop:2} are not directly tractable and need to be reformulated in this regard. This is stated in the following result:
\begin{theorem}\label{th:sw}
  Let $0<T_{min}\le T_{max}<\infty$. Then, the following statements are equivalent:
  \begin{enumerate}[(a)]
    \item There exist matrices $P_i\in\mathbb{S}^{n+m}_{\succ0}$, $K_{ij}\in\mathbb{R}^{m\times (n+m)}$, $i,j=1,\ldots,N$, and a nonnegative matrix $\Pi\in\mathbb{R}^{N\times N}$ verifying $\mathds{1}_N^T\Pi=\mathds{1}_N^T$ such that the conditions in Proposition \ref{prop:2} hold.
        %
    \item There exist some differentiable matrix-valued functions $S_i:[0,T_{max}]\mapsto\mathbb{S}^{n+m}$, matrices $P_i\in\mathbb{S}_{\succ0}^{n+m}$, $K_i\in\mathbb{R}^{m\times(n+m)}$, $i=1,\ldots,N$, a nonnegative matrix $\Pi\in\mathbb{R}^{N\times N}$ verifying $\mathds{1}_N^T\Pi=\mathds{1}_N^T$  and a scalar $\eps>0$ such that the conditions
\begin{equation}
  -\dot{S}_i(\tau)+\bar A_i^TS_i(\tau)+S_i(\tau)\bar A_i\preceq0
\end{equation}
\begin{equation}
  -P_i+S_i(\theta)+\eps I\preceq0
\end{equation}
and
\begin{equation}
  \sum_{j=1}^N\pi_{ji}\bar J_{j,i}^TP_j\bar J_{j,i}-S_i(0)\preceq0
\end{equation}
hold for all $i=1,\ldots,N$, all $\tau\in[0,T_{max}]$ and all $\theta\in[T_{min},T_{max}]$.
\item There exist some differentiable matrix-valued functions $\tilde S_i:[0,T_{max}]\mapsto\mathbb{S}^{n+m}$, matrices $\tilde P_i\in\mathbb{S}_{\succ0}^{n+m}$, $U_{i,j}\in\mathbb{R}^{m\times (n+m)}$, $i,j=1,\ldots,N$, a nonnegative matrix $\Pi\in\mathbb{R}^{N\times N}$ verifying $\mathds{1}_N^T\Pi=\mathds{1}_N^T$  and a scalar $\eps>0$ such that the conditions
\begin{equation}
  \dot{\tilde S}_i(\tau)+\tilde S_i(\tau)\bar A_i^T+\bar A_i\tilde S_i(\tau)\preceq0
\end{equation}
\begin{equation}
    \tilde P_i-\tilde S_i(\theta)+\eps I\preceq0
\end{equation}
and
\begin{equation}
\begin{bmatrix}
  -\tilde S_i(0) & V_i^T\\
   V_i & -\diag_{j=1}^N\{\tilde P_j\}
\end{bmatrix}\preceq0
\end{equation}
hold for all $i=1,\ldots,N$ and all $\theta\in[T_{min},T_{max}]$ where $V_i=\col_{j=1}^N\{\pi_{ji}^{1/2}[\bar J_{j,i}^0\tilde{S}_i(0)+\bar J_{j,i}^1U_{j,i}]\}$.
  \end{enumerate}
  Moreover, when the conditions of statement (c) hold, then the conditions of Proposition \ref{prop:2} hold with $K_{j,i}=U_{j,i}\tilde{S}_i(0)^{-1}$, $i,j=1,\ldots,N$, $i\ne j$, and $P_i=\tilde{P}_i^{-1}$, $i=1,\ldots,N$. As a result, the system \eqref{eq:mainsyst} with controller gains $K_{j,i}=U_{j,i}\tilde{S}_i(0)^{-1}$ and min-jumping rule \eqref{eq:rule2} is asymptotically stable for any sequence $\tkZ$ satisfying the range dwell-time condition $t_{k+1}-t_k\in[T_{min},T_{max}]$.
\end{theorem}
\begin{proof}
This result can be proven in the same way as Theorem \ref{th:main_IMP}. 
\end{proof}

\blue{\subsection{Computational results}

\begin{mybox*}
\caption{SOS program associated with Theorem \ref{th:sw}, (c)}\label{box3}
{\vspace{1mm}}
\noindent\fbox{
\parbox{\textwidth}{
Find symmetric matrix-valued polynomial functions $\tilde{S}_i,\Gamma_i,\Delta_i:\mathbb{R}\mapsto\mathbb{S}^{n+m}$, $i=1,\ldots,N$, constant symmetric matrices $\tilde{P}_i$, $i=1,\ldots,N$, constant matrices $U_i\in\mathbb{R}^{m\times(n+m)}$, $i=1,\ldots,N$, a nonnegative matrix $\Pi\in\mathbb{R}^{N\times N}$ such that $\mathds{1}^T\Pi=\mathds{1}^T$ and a scalar $\eps>0$ such that
      \begin{itemize}
        \item $\Gamma_i(\cdot),\Delta_i(\cdot)$, $i=1,\ldots,N$, are SOS matrices,
        \item $P_i-\eps I$, $i=1,\ldots,N$, are SOS matrices,
        \item $-\dot{\tilde{S}}_i(\tau)-\He[\tilde{S}_i(\tau)\bar A_i^T] -\Gamma_i(\tau)g(\tau)$ is an SOS matrix for all $i=1,\ldots,N$,
        \item $-\tilde{P}_i+\tilde{S}_i(\theta)-\eps I -\Delta_i(\theta)h(\theta)$ is an SOS matrix for all $i=1,\ldots,N$,
        \item $\begin{bmatrix}
          \tilde{S}_i(0) & V_i^T\\
          V_i & \diag_{j=1}^N[\tilde{P}_j]
        \end{bmatrix}$ is an SOS matrix for all $i=1,\ldots,N$.
    \end{itemize}}}
\end{mybox*}

\begin{theorem}\label{th:NCS3}
   Let $0<T_{min}\le T_{max}<\infty$ be given. Then, the following statements are equivalent:
   \begin{enumerate}[(a)]
     \item The conditions of Theorem \ref{th:sw}, (a) hold.
     \item There exist some scalars $\eps>0$ and $d\in\mathbb{Z}_{\ge0}$ such that the SOS program in Box \ref{box3} is feasible using polynomial matrices of degree $2d$.
   \end{enumerate}
\end{theorem}
\begin{proof}
  The proof is identical to that of Theorem \ref{th:NCS1} and Theorem \ref{th:NCS2}.
\end{proof}}

\section{Examples}\label{sec:ex}

We provide several \blue{simple} illustrative examples. \blue{Although simple, some of these results are interesting from an educational viewpoint as the impact of the different parts of the system onto its stability is immediately clear from their formulation.} The infinite-dimensional conditions stated in the main results of the paper are checked using the package SOSTOOLS  \cite{sostools3} and the semidefinite programming solver SeDuMi \cite{Sturm:01a}. Examples of sum of squares programs can be found in \cite{Briat:13d,Briat:14f,Briat:15f}.

\subsection{Example 1. Sampled-data control or state jumps}

Let us consider the system \eqref{eq:mainsyst_imp} with the matrices
  \begin{equation}\label{eq:ex1a}
    A=\begin{bmatrix}
      3 & 0\\
      1 & 1
    \end{bmatrix},\ B=\begin{bmatrix}
      0\\1
    \end{bmatrix}
  \end{equation}
  and we define \bblue{two jump matrices}
  \begin{equation}\label{eq:ex1b}
\bar J_1=\begin{bmatrix}
    I_2 & \vline & 0\\
    \hline
    K^1 &\vline& K^2
    \end{bmatrix}\ \textnormal{and}\ \bar J_2=\begin{bmatrix}
      0.7 & 0 & \vline &0\\
      0 & 1.1 & \vline &0\\
      \hline
      0 & 0 &\vline & 1
    \end{bmatrix}.
  \end{equation}
  \blue{The rationale for considering this system lies in the fact that the pair $(A,B)$ is non-stabilizable as the first state is not affected by the sampled-data control input while the second one is. To overcome this situation, we assume the existence of a jump map that can have a stabilizing effect on the first state with the caveat that it slightly destabilizes the second one. Hence, stabilizing the system should be possible provided that one can find suitable controller gains $K^1$ and $K^2$ as well as a suitable jumping rule of the form \eqref{eq:rule1}. It is interesting to see that a game needs to be played between the two jump matrices as applying the first one stabilizes the second state but leaves the dynamics of the first one unchanged (i.e. exponentially increasing) whereas applying the second jump matrix will result in the decrease in magnitude of the first state but an increase of the first one.}

 Applying then the conditions of  Theorem \ref{th:main_IMP}(c) with $T_{min}=10$ms, $T_{max}=50$ms, $\pi_{12}=\pi_{21}=0.1$ with polynomials of order 2 in the associated sum of squares conditions yields the matrices of the min-jumping rule given by
 \begin{equation*}
   P_1=\begin{bmatrix}
      0.1184  &  0.0184   & 0.0023\\
    0.0184  &  0.5032   & 0.0183\\
    0.0023 &   0.0183    &0.0027
   \end{bmatrix},\
    P_2=\begin{bmatrix}
     0.0866   & 0.0877   & 0.0108\\
    0.0877  &  1.3107   & 0.1124\\
    0.0108  &  0.1124   & 0.0142
   \end{bmatrix}
  \end{equation*}
  and the following controller gain
 \begin{equation}
   \begin{bmatrix}
     K^1 & \vline & K^2
   \end{bmatrix}=\begin{bmatrix}
     -0.9622  & -7.7351  & \vline &  -0.0260
   \end{bmatrix}.
 \end{equation}
Simulation results are depicted in Fig.~\ref{fig:ex1} where we can see that the co-designed sampled-data jump rule and state-feedback control law are effectively able to stabilize the open-loop unstable system.
\begin{figure}
  \centering
  \includegraphics[width=0.8\textwidth]{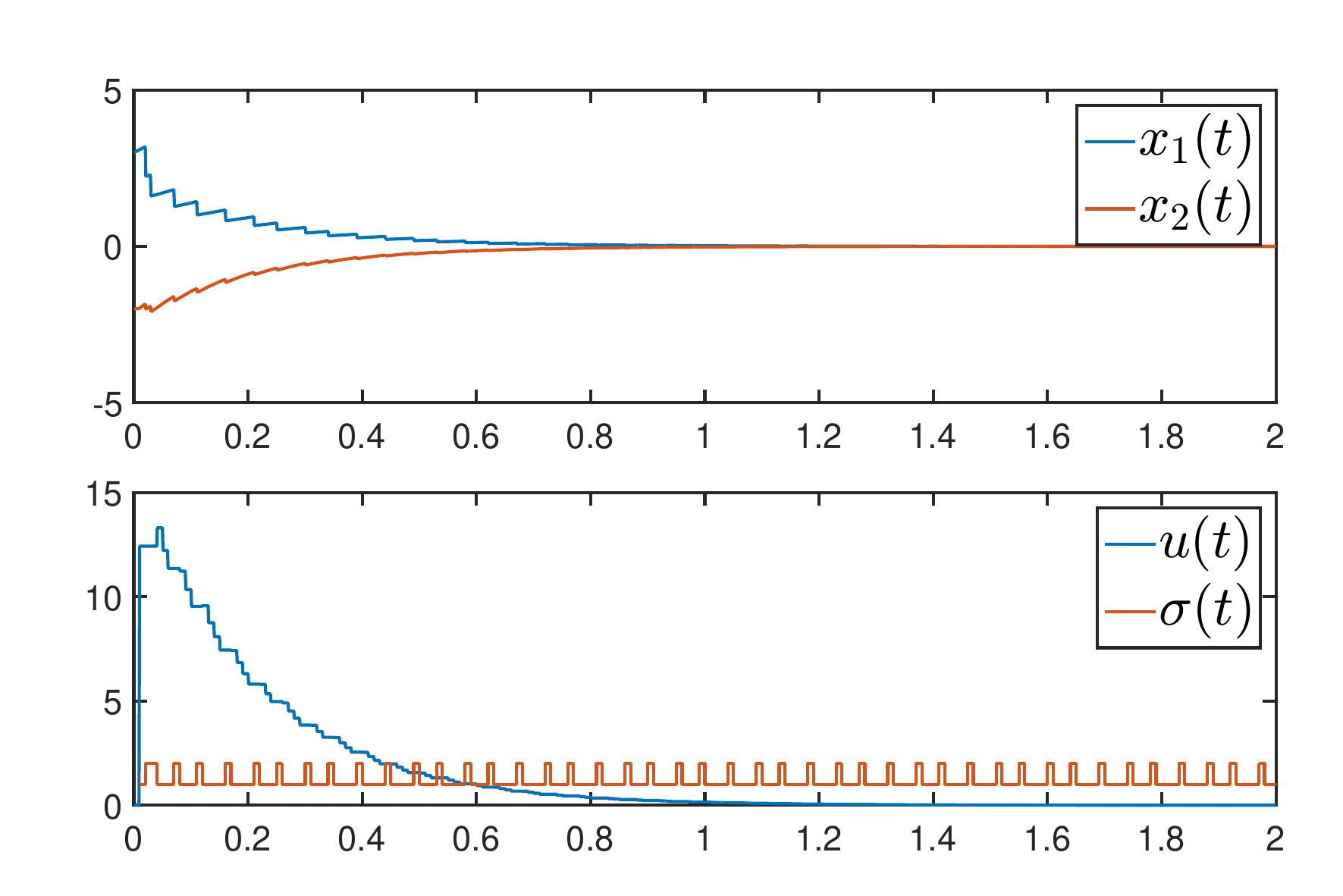}
  \caption{Simulation results for the system \eqref{eq:mainsyst_imp}, \eqref{eq:ex1a}, \eqref{eq:ex1b}.}\label{fig:ex1}
\end{figure}

\subsection{Example 2. Game vs. two jump matrices}

We consider here the system \eqref{eq:mainsyst_imp} with the matrices
\begin{equation}\label{eq:ex2a}
A=\begin{bmatrix}
2 & 3\\1 & 1
\end{bmatrix},\ J_1=\begin{bmatrix}
  1 & 0\\
  0 & 0.8
\end{bmatrix}\ \textnormal{and}\ J_2=\begin{bmatrix}
  0.7 & 0\\
  0 & 1
\end{bmatrix}
\end{equation}
where we can see that, as before, a game needs to be played between the two jump matrices. Applying the conditions of Theorem \ref{th:main_IMP}(c) (adapted to the case where no state-feedback controller has to be designed) with $\pi_{11}=\pi_{22}=0.9$, $T_{min}=T_{max}=0.02$ yields the matrices
\begin{equation}
  P_1=\begin{bmatrix}
    25.5386 &   6.3780\\
    6.3780   & 6.6746
  \end{bmatrix}\ \textnormal{and}\ P_2=\begin{bmatrix}
  2.8886&    2.8549\\
    2.8549  & 20.6927
  \end{bmatrix}.
\end{equation}
The simulation results in Fig.~\ref{fig:ex2} illustrate the ability of the min-switching rule to efficiently stabilize the system.
\begin{figure}
  \centering
  \includegraphics[width=0.8\textwidth]{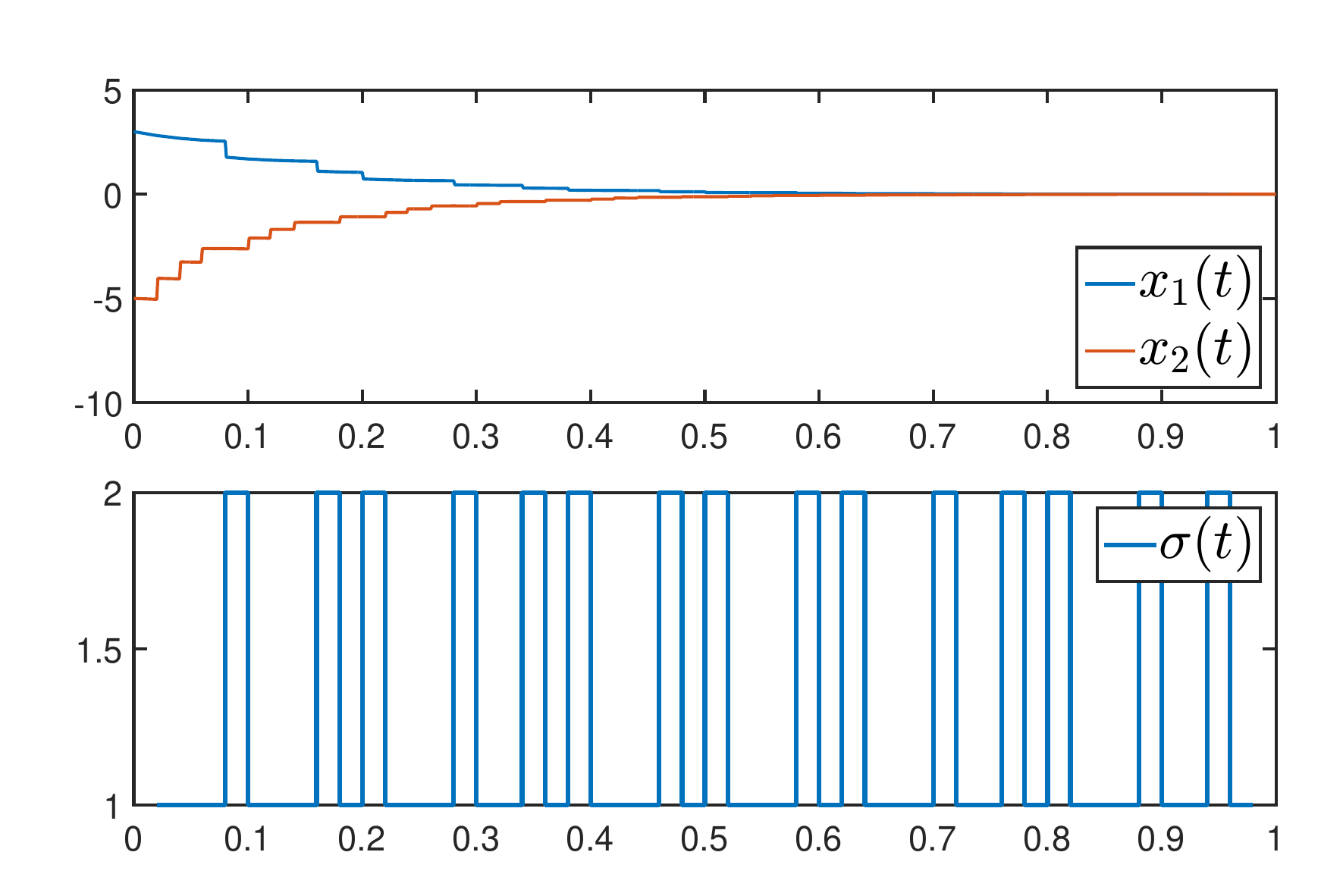}
  \caption{Simulation results for the system \eqref{eq:mainsyst_imp}, \eqref{eq:ex2a}.}\label{fig:ex2}
\end{figure}

\subsection{Example 3. Actuator switching}

Let us consider the system \eqref{eq:mainsyst} with the matrices
\begin{equation}\label{eq:ex3a}
  A_1=A_2=\begin{bmatrix}
  3 & 0\\1 & 1
  \end{bmatrix},\ B_1=\begin{bmatrix}
    6 & 0\\
    0 & 0
  \end{bmatrix}\ \textnormal{and}\ B_2=\begin{bmatrix}
 0 & 0\\
 0 & 4
  \end{bmatrix}.
\end{equation}
This system represents a system for which only one actuator can be updated at a time while the other maintains its previous control input. The first one can act on both states (i.e. the pair $(A,B_1)$ is controllable) whereas the second one can only act on the second state. Applying the conditions of Theorem \ref{th:sw}(c) with $\pi_{11}=\pi_{22}=0.1$, $T_{min}=10$ms, $T_{max}=50$ms, with polynomials of order 2 in the associated sum of squares conditions yields the matrices
\begin{equation*}
 P_1=10\begin{bmatrix}
  0.63 &  -0.06 &  -2.14 &  -0.11\\
   -0.06  &  3.51 &   0.03 &  -7.64\\
   -2.14  &  0.03  &  7.45 &   0.59\\
   -0.11&   -7.64  &  0.59  &  203.99
  \end{bmatrix}
\end{equation*}
\begin{equation*}
  P_2=10\begin{bmatrix}
     0.43 &  -0.05 &  -1.53&   -0.08\\
   -0.05  &  4.03  &  0.05&   -8.53\\
   -1.53  &  0.05 &   191.30  &  0.61\\
   -0.08  & -8.53  &  0.61&    23.52
\end{bmatrix}
\end{equation*}
together with
\begin{equation*}
  \begin{array}{lcl}
    K_{1,1}&=&\begin{bmatrix}
      -3.4332&   -0.0457  & -0.0061 &  -0.0003
    \end{bmatrix},\vspace{1mm}\\
    K_{1,2}&=&\begin{bmatrix}
   -3.4160 &  -0.0516 &  -0.0001 &  -0.0033
    \end{bmatrix},\vspace{1mm}\\
    K_{2,1}&=&\begin{bmatrix}
  -0.4073 &  -2.1272  &  0.0076&   -0.0004
    \end{bmatrix},\vspace{1mm}\\
    K_{2,2}&=&\begin{bmatrix}
     -0.4323  & -2.1198 &   0.0014   & 0.0043
    \end{bmatrix}.
  \end{array}
\end{equation*}
The simulation results depicted in Fig.~\ref{fig:ex3} demonstrate the stabilization effect of the proposed co-design approach.
\begin{figure}
  \centering
  \includegraphics[width=0.8\textwidth]{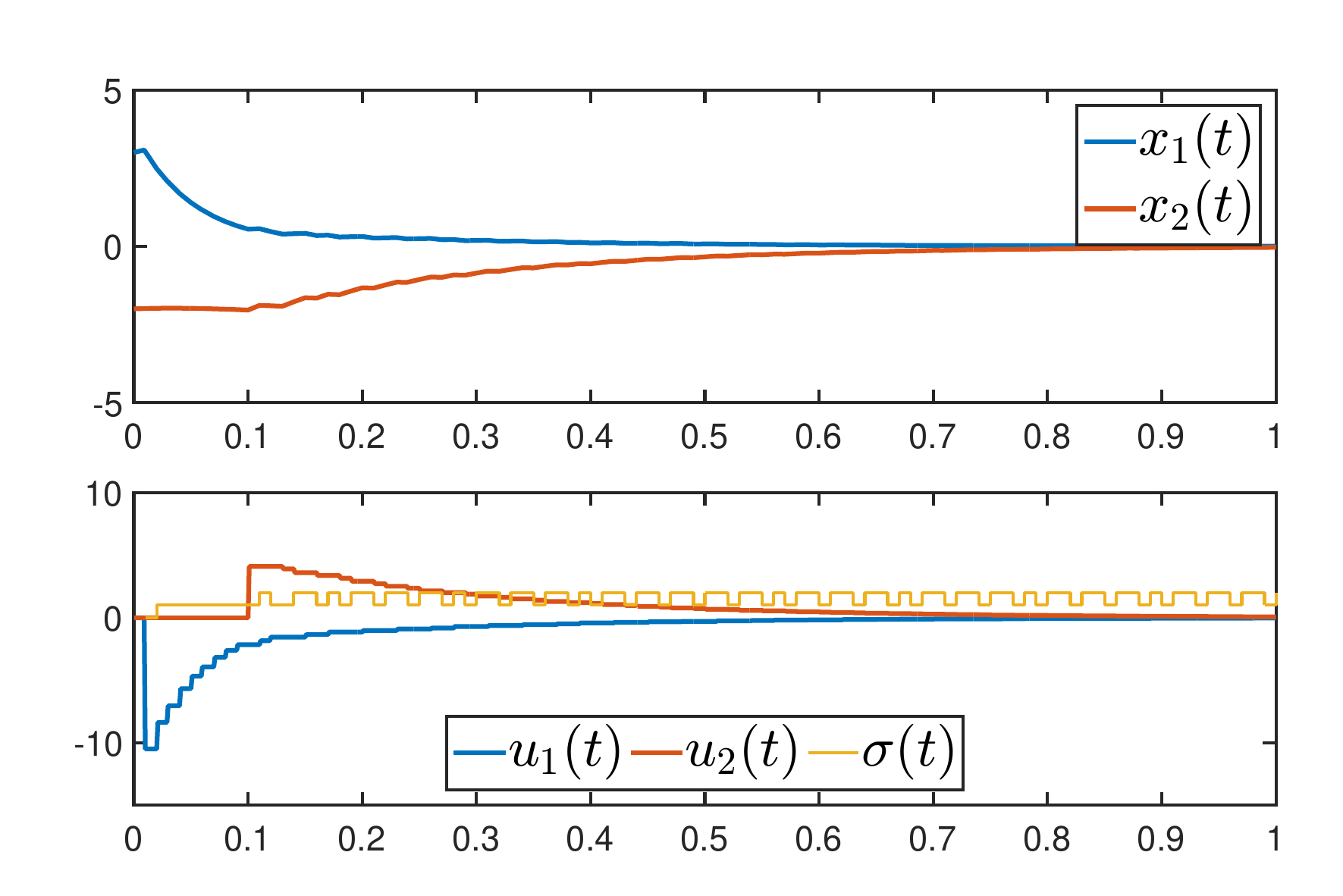}
  \caption{Simulation results for the system \eqref{eq:mainsyst}, \eqref{eq:ex3a}.}\label{fig:ex3}
\end{figure}

\blue{\subsection{Example 5. Co-design for sampled-data systems}

Let us consider here the continuous-time system
\begin{equation}\label{eq:dsdsdsd}
  \dot{x}=\begin{bmatrix}
    0 & 0\\1 & 0
  \end{bmatrix}x+\begin{bmatrix}
    0\\
    10
  \end{bmatrix}u
\end{equation}
for which we aim to design a sampled-data state-feedback control law together with a decision rule on whether to update the value of the control input or keep the previous one when a sampled measurement arrives. This can be modeled using the augmented system
\begin{equation}
  \dot{z}=\begin{bmatrix}
    0 & 0 & 0\\1 & 0 &0\\0 & 10 & 0
  \end{bmatrix}z
\end{equation}
where $z=\col(x,u)$ which is subject to the jump matrices $J_1=I_3$ and
\begin{equation}
  J_2=\begin{bmatrix}
    I_2 & 0\\
    K_1 & K_2
  \end{bmatrix}
\end{equation}
where the controller matrices $K_1,K_2$ need to be designed. Choosing now $\Pi=\begin{bmatrix}
  0.2 & 0.5\\
  0.8 & 0.5
\end{bmatrix}$ and $T_{min}=0.01$, polynomials of degree 4, we can show that  we can stabilize such a system with aperiodic measurements such that $T_k\in[0.01,0.13]$ with the controller matrices
\begin{equation}
  K_1=\begin{bmatrix}
     -0.0675  & -0.0501
  \end{bmatrix}\ \textnormal{and }K_2=-0.2021.
\end{equation}
The SOS program has 1803 primal and 480 dual variables. It takes less than 3 seconds to be solved. We also get the matrices
\begin{equation}
  P_1=\begin{bmatrix}
      0.7947 &   -0.3382  & -0.0253\\
   -0.3382&    1.7296  & -0.0530\\
   -0.0253  & -0.0530&    1.0493
  \end{bmatrix},\ P_2=\begin{bmatrix}
     0.7968 &  -0.3388&   -0.0322\\
   -0.3388  &  1.7313&   -0.0527\\
   -0.0322 &  -0.0527   & 0.0135
  \end{bmatrix}
\end{equation}

\begin{figure}
  \centering
  \includegraphics[width=0.8\textwidth]{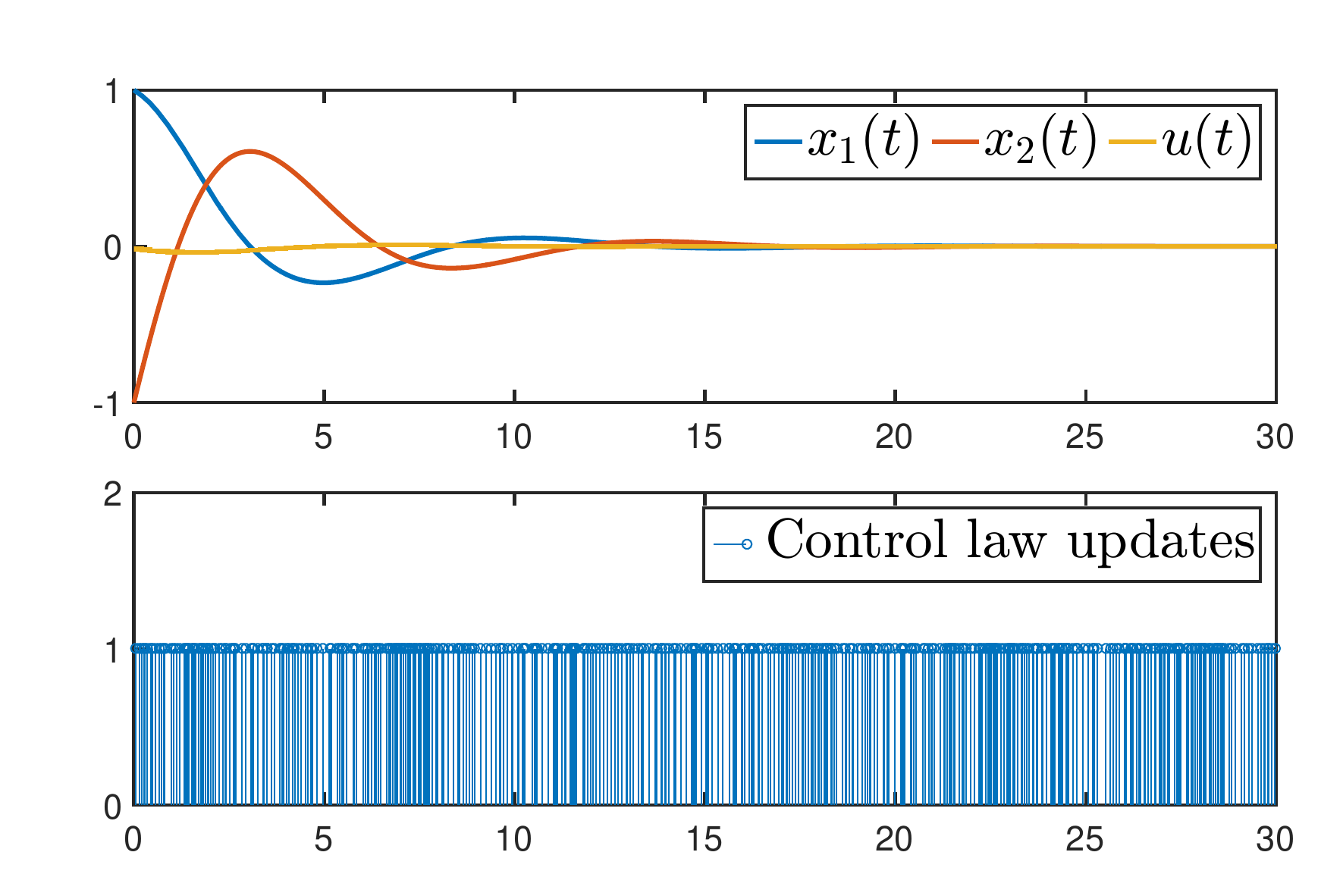}
  \caption{Simulation results for the system \eqref{eq:mainsyst}, \eqref{eq:dsdsdsd}.}
\end{figure}

}

\blue{\subsection{Example 5. Co-design for switched systems}

Let us start with a switched-system of the form \eqref{eq:mainsyst} with the matrices
\begin{equation}\label{eq:ex4a}
  A_1=\begin{bmatrix}
    1 & 1 &  0\\
    1  & 2 &  0\\
     1 & 2   &  2
  \end{bmatrix},B_1=\begin{bmatrix}
   0\\
    0\\
     10
  \end{bmatrix},A_2=\begin{bmatrix}
     1&    2   &  2\\
     0& 2  &   1\\
     0&    1   &  2
  \end{bmatrix}, B_2=\begin{bmatrix}
   10\\
     0\\
    0
  \end{bmatrix}
  \end{equation}
and
\begin{equation}
  A_3=\begin{bmatrix}
     1&  0   & 2\\
     1&    2&    3\\
     1   & 0 &    2
  \end{bmatrix}, B_3=\begin{bmatrix}
0\\
10\\
0
  \end{bmatrix}.
  \end{equation}
The spectrum of all the above matrices is located in the right-half plane. Moreover, the controllability matrix of each subsystem has rank one, which means that only one state can be controlled at a time, namely, state 3,1 and 2 for the subsystem 1,2 and 3, respectively. The goal is then to find a sampled-data control law that allows for the stabilization of the controllable state and a switching law allowing to stabilize the overall system. \bblue{We hence, define the augmented system
\begin{equation}
  \dot{z}=\begin{bmatrix}
    A_\sigma & B_\sigma\\
    0 & 0
  \end{bmatrix}z
\end{equation}
where $z=\col(x,u)$.} We also select the following parameters $d=2$, $T_{min}=0.02$, $T_{max}=0.05$  and $$\Pi=\begin{bmatrix}
  0.3530  &  0.2659  &  0.1112\\
    0.1962&    0.1216  &  0.4766\\
    0.4508   & 0.6126   & 0.4122
\end{bmatrix}$$

The SDP problem has 2200 primal and 738 dual variables and is solved in less than 2 seconds. The following Lyapunov matrices are obtained:
\begin{equation}
  P_1=\begin{bmatrix}
165.6800   & 38.2540  & -38.7790  & -13.5960\\
   38.2540  & 201.3500  & -25.3810  & -40.0490\\
  -38.7790  & -25.3810   & 37.4760  & -52.1520\\
  -13.5960 &  -40.0490  &  -52.1520  & 140.7800
  \end{bmatrix},
\end{equation}

\begin{equation}
  P_2=\begin{bmatrix}
288.9500   & 87.3970  &  -40.7430 & -352.2100\\
   87.3970 &  227.9900 &  -36.3720  & -99.6230\\
  -40.7430 &  -36.3720 &   20.6920  &   7.6852\\
 -352.2100 &  -99.6230  &   7.6852 &  602.5600
  \end{bmatrix},
\end{equation}

\begin{equation}
  P_3=\begin{bmatrix}
200.3700  & 91.3040  &-46.6470  &-62.3100\\
   91.3040 & 372.8700  &-34.7260 &-408.0100\\
  -46.6470 & -34.7260  & 20.7850   & 11.0040\\
  -62.3100 &-408.0100 &  11.0040  & 529.1500\\
  \end{bmatrix}
\end{equation}
as well as the following state-feedback gains
\begin{equation}
  \begin{array}{rcl}
    K_{1,1}&=&\begin{bmatrix}
      -0.5011 &  -0.3758 &  -2.1639 &   0.0015
    \end{bmatrix},\\K_{1,2}&=&\begin{bmatrix}
       -0.4855   &-0.3760   &-2.1325  &  0.0073
    \end{bmatrix},\\
    K_{1,3}&=&\begin{bmatrix}
       -0.4921&   -0.3463 &  -2.1051   & 0.0203
    \end{bmatrix},\\K_{2,1}&=&\begin{bmatrix}
      -1.5984 &  -0.3747&   -3.4286  & -0.0003
    \end{bmatrix},\\
    K_{2,2}&=&\begin{bmatrix}
      -1.6345&   -0.3783&   -3.5077  & -0.0163
    \end{bmatrix},\\K_{2,3}&=&\begin{bmatrix}
      -1.5837&   -0.3332&   -3.3568&    0.0272
    \end{bmatrix},\\
    K_{3,1}&=&\begin{bmatrix}
      -0.2200&   -1.2307&   -2.0104   & 0.0026
    \end{bmatrix},\\K_{3,2}&=&\begin{bmatrix}
      -0.2145 &  -1.2328 &  -2.0120 &   0.0030
    \end{bmatrix},\\
    K_{3,3}&=&\begin{bmatrix}
      -0.2212  & -1.2371 &  -2.0284 &  -0.0038
    \end{bmatrix}.
  \end{array}
\end{equation}

The simulation results depicted in Fig.~\ref{fig:ex4} demonstrate the stabilization effect of the proposed co-design approach.
\begin{figure}
  \centering
  \includegraphics[width=0.8\textwidth]{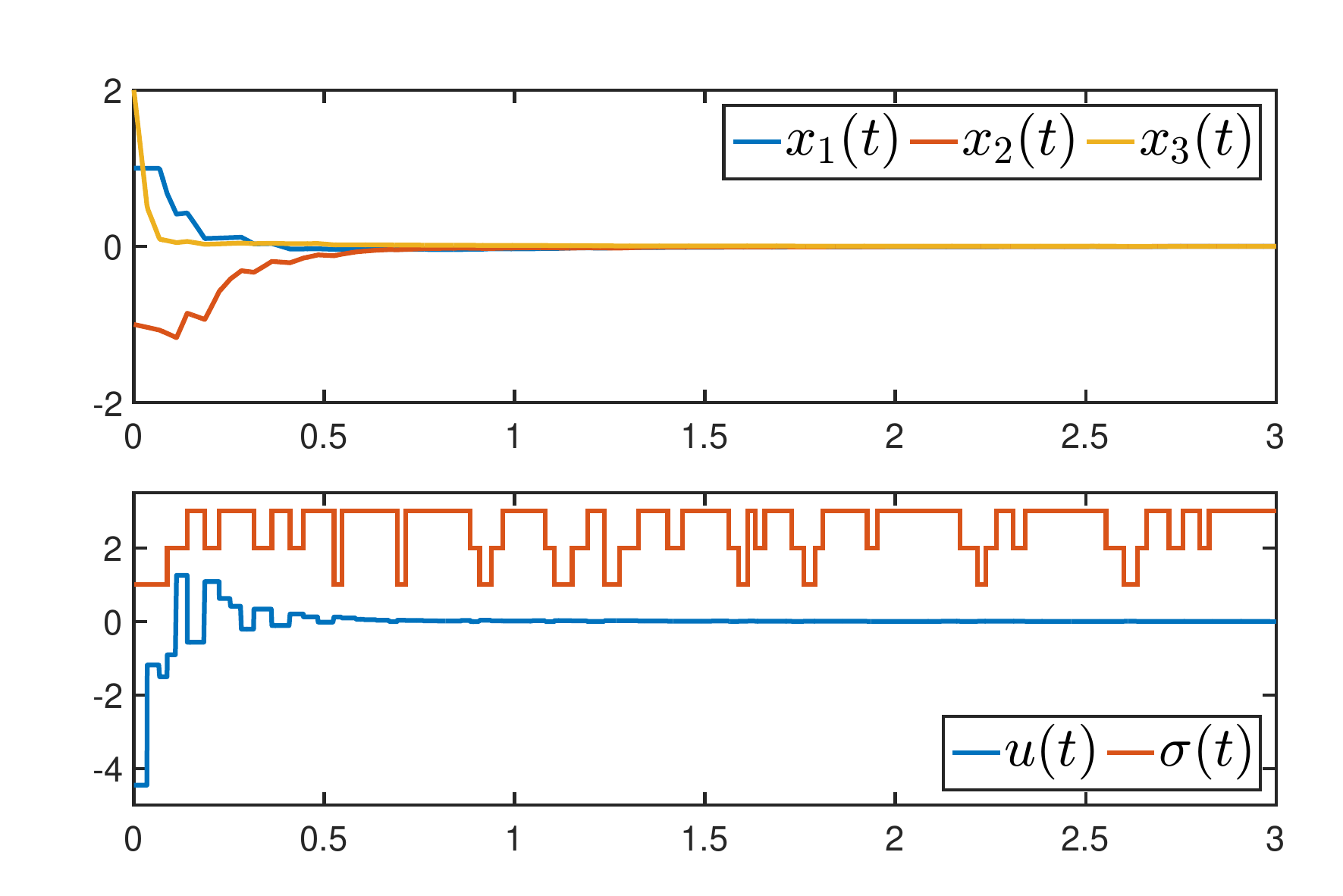}
  \caption{Simulation results for the system \eqref{eq:mainsyst}, \eqref{eq:ex4a}.}\label{fig:ex4}
\end{figure}}

\blue{\section{Conclusion}

A co-design approach for the simultaneous design of a sampled-data control law and a min-jumping rule for jump systems subject to sampled measurements has been provided. The approach naturally extends to address the co-design problem for switched systems subject to sampled measurements. The stability and stabilization conditions are expressed as infinite-dimensional Lyapunov-Metzler conditions that can reduce to infinite-dimensional semidefinite programs when the Lyapunov-Metzler variable is set to a specific value. As the conditions are infinite-dimensional, they cannot be directly checked. Relaxed results based on sum of squares (SOS) are provided and result in finite-dimensional semidefinite programs when the degree of the polynomials are fixed. Converse results are then given to demonstrate that SOS relaxations are, in fact, non conservative if we allow the degree of the polynomials to be arbitrarily large. Numerical results tend to suggest that satisfying results are already obtained for low polynomial degrees. Examples are given for illustrations. Potential extensions include the consideration of system uncertainties and the inclusion of performance measures such that the $L_2$-performance measure. These extensions are, in fact, rather straightforward due to the affine dependence of the conditions in terms of the matrices of the system.}


\begin{thebibliography}{10}
\providecommand{\url}[1]{#1}
\csname url@samestyle\endcsname
\providecommand{\newblock}{\relax}
\providecommand{\bibinfo}[2]{#2}
\providecommand{\BIBentrySTDinterwordspacing}{\spaceskip=0pt\relax}
\providecommand{\BIBentryALTinterwordstretchfactor}{4}
\providecommand{\BIBentryALTinterwordspacing}{\spaceskip=\fontdimen2\font plus
\BIBentryALTinterwordstretchfactor\fontdimen3\font minus
  \fontdimen4\font\relax}
\providecommand{\BIBforeignlanguage}[2]{{%
\expandafter\ifx\csname l@#1\endcsname\relax
\typeout{** WARNING: IEEEtran.bst: No hyphenation pattern has been}%
\typeout{** loaded for the language `#1'. Using the pattern for}%
\typeout{** the default language instead.}%
\else
\language=\csname l@#1\endcsname
\fi
#2}}
\providecommand{\BIBdecl}{\relax}
\BIBdecl

\bibitem{Anta:10}
A.~Anta and P.~Tabuada, ``To sample or not to sample: Self-triggered control
  for nonlinear systems,'' \emph{IEEE Transactions on Automatic Control}, vol.
  55(9), pp. 2030--2042, 2010.

\bibitem{Seuret:11b}
A.~Seuret and C.~Prieur, ``Event-triggered sampling algorithms based on a
  lyapunov function,'' in \emph{50th IEEE Conference on Decision and Control},
  Orlando, USA, 2011, pp. 6128--6133.

\bibitem{Liberzon:03}
D.~Liberzon, \emph{Switching in Systems and Control}.\hskip 1em plus 0.5em
  minus 0.4em\relax New York: Birkh{\"{a}}user, 2003.

\bibitem{Allerhand:13}
L.~I. Allerhand and U.~Shaked, ``{Robust state-dependent switching of linear
  systems with dwell-time},'' \emph{IEEE Transactions on Automatic Control},
  vol. 58(4), pp. 994--1001, 2013.

\bibitem{Hristu:99}
D.~Hristu and K.~Morgansen, ``Limited communication control,'' \emph{Systems \&
  Control Letters}, vol.~37, pp. 193--205, 1999.

\bibitem{Peters:15}
E.~G.~W. Peters, D.~E. Quevedo, and M.~Fu, ``Co-design for control and
  scheduling over wireless industrial control networks,'' in \emph{54th IEEE
  Conference on Decision and Control}, Osaka, Japan, 2015, pp. 2459--2464.

\bibitem{Morse:96}
A.~S. Morse, ``Supervisory control of families of linear set-point controllers
  - {P}art 1: {E}xact matching,'' \emph{IEEE Transactions on Automatic
  Control}, vol. 41(10), pp. 1413--1431, 1996.

\bibitem{Prieur:11}
C.~Prieur and A.~R. Teel, ``{Uniting local and global output feedback
  controllers},'' \emph{{IEEE Transactions on Automatic Control}}, vol. 56(7),
  pp. 1636--1649, 2011.

\bibitem{Geromel:06b}
J.~C. Geromel and P.~Colaneri, ``Stability and stabilization of continuous-time
  switched linear systems,'' \emph{{SIAM} Journal on Control and Optimization},
  vol. 45(5), pp. 1915--1930, 2006.

\bibitem{Geromel:06d}
------, ``Stability and stabilization of discrete-time switched systems,''
  \emph{International Journal of Control}, vol. 79(7), pp. 719--728, 2006.

\bibitem{Kruszewski:11}
A.~Kruszewski, R.~Bourdais, and W.~Perruquetti, ``Converging algorithm for a
  class of {BMI} applied on state-dependent stabilization of switched
  systems,'' \emph{Nonlinear Analysis: Hybrid Systems}, vol.~5, pp. 647--654,
  2011.

\bibitem{Parrilo:00}
P.~Parrilo, ``Structured semidefinite programs and semialgebraic geometry
  methods in robustness and optimization,'' Ph.D. dissertation, California
  Institute of Technology, Pasadena, California, 2000.

\bibitem{sostools3}
A.~Papachristodoulou, J.~Anderson, G.~Valmorbida, S.~Prajna, P.~Seiler, and
  P.~A. Parrilo, \emph{{SOSTOOLS}: Sum of squares optimization toolbox for
  {MATLAB} v3.00}, 2013.

\bibitem{Putinar:93}
M.~Putinar, ``Positive polynomials on compact semi-algebraic sets,''
  \emph{Indiana Univ. Math. J.}, vol.~42, no.~3, pp. 969--984, 1993.

\bibitem{Briat:11l}
C.~Briat and A.~Seuret, ``A looped-functional approach for robust stability
  analysis of linear impulsive systems,'' \emph{Systems \& Control Letters},
  vol. 61(10), pp. 980--988, 2012.

\bibitem{Briat:12h}
------, ``Convex dwell-time characterizations for uncertain linear impulsive
  systems,'' \emph{{IEEE} Transactions on Automatic Control}, vol. 57(12), pp.
  3241--3246, 2012.

\bibitem{Briat:13b}
------, ``Affine minimal and mode-dependent dwell-time characterization for
  uncertain switched linear systems,'' \emph{{IEEE} Transactions on Automatic
  Control}, vol.~58, no.~5, pp. 1304--1310, 2013.

\bibitem{Briat:13d}
C.~Briat, ``Convex conditions for robust stability analysis and stabilization
  of linear aperiodic impulsive and sampled-data systems under dwell-time
  constraints,'' \emph{Automatica}, vol. 49(11), pp. 3449--3457, 2013.

\bibitem{Briat:14f}
------, ``Convex conditions for robust stabilization of uncertain switched
  systems with guaranteed minimum and mode-dependent dwell-time,''
  \emph{Systems \& Control Letters}, vol.~78, pp. 63--72, 2015.

\bibitem{Briat:15i}
------, ``Stability analysis and stabilization of stochastic linear impulsive,
  switched and sampled-data systems under dwell-time constraints,''
  \emph{Automatica}, vol.~74, pp. 279--287, 2016.

\bibitem{Hetel:13}
L.~Hetel and E.~Fridman, ``Robust sampled-data control of switched affine
  systems,'' \emph{IEEE Transactions on Automatic Control}, vol. 58(11), pp.
  2922--2928, 2013.

\bibitem{Deaecto:13b}
G.~S. Deaecto, M.~Souza, and J.~C. Geromel, ``Chattering free control of
  continuous-time switched systems,'' \emph{IET Control Theory and
  Applications}, vol. 8(5), pp. 348--354, 2013.

\bibitem{Briat:11d}
Y.~Ariba, C.~Briat, and K.~H. Johansson, ``Simple conditions for ${L_2}$
  stability and stabilization of networked control systems,'' in \emph{18th
  {IFAC} World Congress}, Milano, Italy, 2011, pp. 96--101.

\bibitem{Yuan:16}
C.~Yuan and F.~Wu, ``Delay scheduled impulsive control for networked control
  systems,'' \emph{Transactions on Control of Network Systems (in press)},
  2016.

\bibitem{Fridman:04}
E.~Fridman, A.~Seuret, and J.~P. Richard, ``Robust sampled-data stabilization
  of linear systems: An input delay approach,'' \emph{Automatica}, vol.~40, pp.
  1441--1446, 2004.

\bibitem{Mirkin:07}
L.~Mirkin, ``Some remarks on the use of time-varying delay to model
  sample-and-hold circuits,'' \emph{IEEE Transactions on Automatic Control},
  vol. 52(6), pp. 1109--1112, 2007.

\bibitem{Fujioka:09b}
H.~Fujioka, ``Stability analysis of systems with aperiodic sample-and-hold
  devices,'' \emph{Automatica}, vol.~45, pp. 771--775, 2009.

\bibitem{Kao:14}
C.-Y. Kao and D.-R. Wu, ``On robust stability of aperiodic sampled-data systems
  - an integral quadratic constraint approach,'' in \emph{American Control
  Conference}, Portland, USA, 2014, pp. 4871--4876.

\bibitem{Fujioka:09a}
H.~Fujioka, ``A discrete-time approach to stability analysis of systems with
  aperiodic sample-and-hold devices,'' \emph{IEEE Transactions on Automatic
  Control}, vol. 54(10), pp. 2440--2445, 2009.

\bibitem{Oishi:10}
Y.~Oishi and H.~Fujioka, ``Stability and stabilization of aperiodic
  sampled-data control systems using robust linear matrix inequalities,''
  \emph{Automatica}, vol.~46, pp. 1327--1333, 2010.

\bibitem{Goebel:09}
R.~Goebel, R.~G. Sanfelice, and A.~R. Teel, ``Hybrid dynamical systems,''
  \emph{{IEEE} Control Systems Magazine}, vol. 29(2), pp. 28--93, 2009.

\bibitem{Naghshtabrizi:08}
P.~Naghshtabrizi, J.~P. Hespanha, and A.~R. Teel, ``Exponential stability of
  impulsive systems with application to uncertain sampled-data systems,''
  \emph{Systems \& Control Letters}, vol.~57, pp. 378--385, 2008.

\bibitem{Seuret:12}
A.~Seuret, ``A novel stability analysis of linear systems under asynchronous
  samplings,'' \emph{Automatica}, vol. 48(1), pp. 177--182, 2012.

\bibitem{Briat:14b}
A.~Seuret and C.~Briat, ``Stability analysis of uncertain sampled-data systems
  with incremental delay using looped-functionals,'' \emph{Automatica},
  vol.~55, pp. 274--278, 2015.

\bibitem{Briat:16c}
C.~Briat, ``Dwell-time stability and stabilization conditions for linear
  positive impulsive and switched systems,'' \emph{Nonlinear Analysis: Hybrid
  Systems}, vol.~24, pp. 198--226, 2017.

\bibitem{Goebel:12}
R.~Goebel, R.~G. Sanfelice, and A.~R. Teel, \emph{Hybrid Dynamical Systems.
  Modeling, Stability, and Robustness}.\hskip 1em plus 0.5em minus 0.4em\relax
  Princeton University Press, 2012.

\bibitem{Hespanha:07}
J.~P. Hespanha, P.~Naghshtabrizi, and Y.~Xu, ``A survey of recent results in
  networked control systems,'' \emph{Proceedings of the {IEEE}}, vol. 95(1),
  pp. 138--162, 2007.

\bibitem{Zhang:16c}
H.~Zhang, G.~Feng, H.~Yan, and Q.~Chen, ``Sampled-data control of nonlinear
  networked systems with time-delay and quantization,'' \emph{International
  Journal of Robust and Nonlinear Control}, vol.~26, pp. 919--933, 2016.

\bibitem{Heemels:17}
W.~P. M.~H. Heemels, A.~Kundu, and J.~Daafouz, ``On {L}yapunov-{M}etzler
  inequalities and {S}-procedure characterisations for the stabilisation of
  switched linear systems,'' \emph{IEEE Transactions on Automatic Control},
  vol. to appear, 2017.

\bibitem{Allerhand:11}
L.~I. Allerhand and U.~Shaked, ``Robust stability and stabilization of linear
  switched systems with dwell time,'' \emph{{IEEE} Transactions on Automatic
  Control}, vol. 56(2), pp. 381--386, 2011.

\bibitem{Briat:15f}
C.~Briat, ``Theoretical and numerical comparisons of looped functionals and
  clock-dependent {L}yapunov functions - {T}he case of periodic and
  pseudo-periodic systems with impulses,'' \emph{International Journal of
  Robust and Nonlinear Control}, vol.~26, pp. 2232--2255, 2016.

\bibitem{Xiang:15a}
W.~Xiang, ``On equivalence of two stability criteria for continuous-time
  switched systems with dwell time constraint,'' \emph{Automatica}, vol.~54,
  pp. 36--40, 2015.

\bibitem{Xiang:16}
------, ``Necessary and sufficient condition for stability of switched
  uncertain linear systems under dwell-time constraint,'' \emph{IEEE
  Transactions on Automatic Control}, vol. 61(11), pp. 3619--3624, 2016.

\bibitem{Handelman:88}
D.~Handelman, ``Representing polynomials by positive linear functions on
  compact convex polyhedra,'' \emph{Pacific Journal of Mathematics}, vol.
  132(1), pp. 35--62, 1988.

\bibitem{Scherer:06}
C.~W. Scherer and C.~W.~J. Hol, ``Matrix sum-of-squares relaxations for robust
  semi-definite programs,'' \emph{Mathematical Programming: Series B}, vol.
  107, pp. 189--211, 2006.

\bibitem{Kamyar:15}
R.~Kamyar and M.~M. Peet, ``Polynomial optimization with applications to
  stability analysis and control - alternatives to sum of squares,''
  \emph{Discrete and Continuous Dynamical Systems Series B}, vol. 20(8), pp.
  2383--2417, 2015.

\bibitem{Le:18}
C.-T. {L\^e} and T.-H.-B. Du', ``Handelman's {P}ositivstellensatz for
  polynomial matrices positive definite on polyhedra,'' \emph{Positivity},
  vol.~22, pp. 449--460, 2018.

\bibitem{Heemels:10b}
W.~P. M.~H. Heemels, N.~{van de Wouw}, R.~H. Gielen, M.~F. Donkers, L.~Hetel,
  S.~Olaru, M.~lazar, J.~daafouz, and S.~Niculescu, ``Comparison of
  overapproximation methods for stability analysis of networked control
  systems,'' in \emph{HSCC}, Stockholm, Sweden, 2010, pp. 181--190.

\bibitem{Gajic:95}
Z.~Gaji{\'{c}} and M.~T.~J. Qureshi, \emph{{L}yapunov Matrix Equation in System
  Stability and Control}.\hskip 1em plus 0.5em minus 0.4em\relax Academic
  Press, 1995.

\bibitem{Sturm:01a}
J.~F. Sturm, ``Using {SEDUMI} $1. 02$, a {M}atlab {T}oolbox for {O}ptimization
  {O}ver {S}ymmetric {C}ones,'' \emph{Optimization Methods and Software},
  vol.~11, no.~12, pp. 625--653, 2001.

\bibitem{Chesi:10b}
G.~Chesi, ``{LMI techniques for optimization over polynomials in control: A
  survey},'' \emph{IEEE Transactions on Automatic Control}, vol. 55(11), pp.
  2500--2510, 2010.

\end{thebibliography}


\end{document}